\documentclass[onefignum,onetabnum]{siamart190516}

\usepackage[utf8]{inputenc}

\usepackage{amsmath}
\usepackage{amsfonts}
\usepackage{amssymb}
\usepackage{mathtools}
\usepackage{algorithm}
\usepackage{algpseudocode}

\usepackage{xcolor}
\usepackage{multicol}

\usepackage{colortbl} 
\usepackage{booktabs}
\renewcommand{\arraystretch}{1.3}

\usepackage{tikz}
\usepackage{pgfplots}
\pgfplotsset{compat=1.15}
\usetikzlibrary{calc}
\usetikzlibrary{patterns}
\usepackage[outline]{contour}
\contourlength{1.5pt}

\definecolor{mygreen}{rgb}{0.2,0.7,0.2}
\definecolor{myblue}{rgb}{0,0.5,1}
\definecolor{mypink}{rgb}{0.7,0.1,0.7}
\definecolor{myorange}{rgb}{1,0.4,0.4}

\usepackage{epstopdf}
\ifpdf
  \DeclareGraphicsExtensions{.eps,.pdf,.png,.jpg}
\else
  \DeclareGraphicsExtensions{.eps}
\fi

\newsiamremark{remark}{Remark}
\newsiamremark{hypothesis}{Hypothesis}
\crefname{hypothesis}{Hypothesis}{Hypotheses}
\newsiamthm{claim}{Claim}

\def\N{\mathbb{N}}
\def\R{\mathbb{R}}

\def\Z{\mathbb{Z}}
\def\Sp{\mathrm{Sp}}
\def\sp{\mathfrak{sp}}
\def\Gr{\mathrm{Gr}}
\def\St{\mathrm{St}}
\def\SpGr{\mathrm{SpGr}}
\def\SpSt{\mathrm{SpSt}}

\def\so{\mathfrak{so}}
\def\O{\mathrm{O}}
\def\GL{\mathrm{GL}}

\def\Sym{\mathrm{Sym}}

\def\D{\mathrm{d}}
\def\qgeo{\mathrm{qgeo}}

\def\Ret{\mathcal{R}}
\def\L{\mathcal{L}}

\DeclareMathOperator{\Exp}{Exp}

\DeclareMathOperator{\expm}{exp_m}
\DeclareMathOperator{\logm}{log_m}
\DeclareMathOperator{\sqrtm}{sqrt_m}
\DeclareMathOperator{\cay}{cay}

\DeclareMathOperator{\tr}{tr}

\DeclareMathOperator{\rank}{rank}
\DeclareMathOperator{\diag}{diag}
\DeclareMathOperator{\Hor}{\mathsf{Hor}}
\newcommand{\hor}{\mathsf{hor}}
\DeclareMathOperator{\Ver}{\mathsf{Ver}}

\DeclareMathOperator{\Ad}{Ad}
\DeclareMathOperator{\grad}{grad}

\DeclareMathOperator{\stab}{stab}
\DeclareMathOperator{\id}{id}

\newcommand{\SP}[1]{\left\langle #1\right\rangle}
\newcommand{\Space}[1]{\mathcal{#1}}

\definecolor{mygreen}{RGB}{47,130,100}

\headers{The real symplectic Stiefel and Grassmann manifolds}{T. Bendokat and R. Zimmermann}

\title{The real symplectic Stiefel and Grassmann manifolds: metrics, geodesics and applications
}

\author{Thomas Bendokat\thanks{Department of Mathematics and Computer Science, University of Southern Denmark~(SDU), Odense, Denmark
  (\email{bendokat@imada.sdu.dk}, \email{zimmermann@imada.sdu.dk}).}
\and Ralf Zimmermann\footnotemark[2]}

\ifpdf
\hypersetup{
  pdftitle={The real symplectic Stiefel and Grassmann manifolds: metrics, geodesics and applications},
  pdfauthor={Thomas Bendokat \and Ralf Zimmermann}
}
\fi

\begin{document}

\maketitle

\begin{abstract}
    The real symplectic Stiefel manifold is the manifold of symplectic bases of symplectic subspaces of a fixed dimension. It features in a large variety of applications in physics and engineering.
    In this work, we study this manifold with the goal of providing theory and matrix-based numerical tools fit for basic data processing.
    Geodesics are fundamental for data processing. However, these are so far unknown.
    Pursuing a Lie group approach, we close this gap and derive 
    efficiently computable formulas for the geodesics both with respect to a 
    natural pseudo-Riemannian metric and a novel Riemannian metric.
    In addition, we provide efficiently computable and invertible retractions.
    Moreover, we introduce the real symplectic Grassmann manifold, i.e., the manifold of symplectic subspaces. Again, we derive efficient formulas for pseudo-Riemannian and Riemannian geodesics and invertible retractions. 
    The findings are illustrated by numerical experiments, where we consider optimization via gradient descent on both manifolds and compare the proposed methods with the state of the art.
    In particular, we treat the `nearest symplectic matrix' problem
    and the problem of optimal data representation via a low-rank symplectic subspace.
    The latter task is associated with the problem of
    finding a `proper symplectic decomposition', which is important in structure-preserving model order reduction of Hamiltonian systems.
\end{abstract}



\section{Introduction}
%

The central object under study in this work is the \emph{real symplectic Stiefel manifold}.
The elements of this matrix manifold are the symplectic bases of fixed order $2k$ of symplectic subspaces in $\R^{2n}$,
\begin{equation*}
    \SpSt(2n,2k):=\left\lbrace U \in \R^{2n \times 2k}\ \middle|\ U^TJ_{2n}U=J_{2k} \right\rbrace,
    \quad  J_{2m} = \begin{bmatrix}
          0 & I_m\\
       -I_m & 0
      \end{bmatrix}, m\in \{n,k\}.
\end{equation*}

 Symplectic structures feature in a large variety of applications in physics and engineering, most prominently in Hamiltonian mechanics \cite{Arnold2001}. Hamiltonian systems are used in applications ranging from molecular dynamics to celestial mechanics, see \cite{Hairer06gni} and references therein. The symplectic Stiefel manifold is of special importance to 
 optimization problems of the form
 \begin{eqnarray}
 \label{eq:basicSympOpt}
 & \min_{U \in \R^{2n \times 2k}}\  & f(U)\\
 \nonumber
 &  \text{s.~t.}                  & U^TJ_{2n}U=J_{2k},
 \end{eqnarray}
 see \cite{GaoSonAbsilStykel2020symplecticoptimization} and references therein, since it allows to tackle such {\em constrained} optimization problems on $\R^{2n\times 2k}$ as {\em unconstrained} optimization problems on $\SpSt(2n,2k)$.
 Fields of applications include the symplectic eigenvalue problem \cite{BhatiaJain2015,SonAbsilGaoStykel2021symplecticeigenvalue} and
projection-based structure-preserving model order reduction for Hamiltonian systems.
Here, an optimization problem of the form
 \eqref{eq:basicSympOpt} appears as the central task of computing a so-called {\em proper symplectic decomposition} \cite{afkham2017structure, peng2016symplectic, buchfink2019symplectic}. 
 
 Riemannian optimization requires that we have explicit formulas for essential geometric quantities at hand, as well as efficient algorithms for practical computations.
 It is understood that the inner geometry of the symplectic Stiefel manifold depends on the chosen metric.
 
 \paragraph{Related work and state of the art}
 Optimization on the real symplectic group is considered in \cite{fiori2011} with respect to a pseudo-Riemannian metric and in \cite{Wang2018riemannian,birtea2020optimization} with respect to a left-invariant Riemannian metric. Quotients of the real symplectic group relating to the real symplectic Stiefel and Grassmann manifolds are briefly introduced in \cite[Subsection 2.1]{sedanomendoza2019}.
 To the best of the authors' knowledge, the first treatment of the real symplectic Stiefel manifold with a view on numerical applications is \cite{GaoSonAbsilStykel2020symplecticoptimization}. The optimization algorithm developed there forms the state of the art.
 The same team of authors compared this method with optimization with respect to a Riemannian metric that stems from a Euclidean metric in the later work \cite{GaoSonAbsilStykel2021symplecticEuclidean}. 
 \paragraph{Main original contributions}
 Starting from the classical real symplectic group, equipped with a bi-invariant pseudo-Riemannian metric,
 we use a Lie group approach to investigate the real symplectic Stiefel manifold.
 This original approach allows us to exploit quotient manifold results from semi-Riemannian geometry \cite{ONeill1983} 
 and enables us to derive the first closed-form expressions for the corresponding pseudo-Riemannian geodesics on $\SpSt(2n,2k)$. 
 Complementary to the pseudo-Riemannian approach, we also introduce a Riemannian metric and derive closed-form expressions for the corresponding Riemannian geodesics.
 In view of optimization tasks, we provide a formula for the gradient associated with the Riemannian metric and efficiently computable and invertible retractions that approximate the pseudo-Riemannian geodesics. 
 
 Moreover, we initiate a study of the manifold of symplectic subspaces that are spanned by symplectic bases, which we term the 
 {\em real symplectic Grassmann manifold} in analogy to the classical Stiefel and Grassmann manifolds. 
 Continuing the quotient manifold approach, we derive corresponding formulas for pseudo-Riemannian and Riemannian geodesics and retractions. 
 We promote symplectic subspaces as the main objects of interest in structure-preserving model order reduction of parameterized Hamiltonian systems.
 
 We illustrate the theoretical findings by means of numerical examples.
 More precisely, we investigate the numerical feasibility of the proposed methods,
 we tackle the nearest symplectic matrix problem on the real symplectic Stiefel manifold
 and compute the optimal symplectic subspace containing a given data matrix
 on the real symplectic Grassmann manifold. The latter problem is directly associated with finding a
 proper symplectic decomposition in the context of structure preserving model reduction.
 We juxtapose the methods' performance with the state of the art.
 \paragraph{Organisation}
 Section \ref{sec:SympGroup} reviews basic facts on the real symplectic group and states its geodesics associated with a natural bi-invariant pseudo-Riemannian metric and a right-invariant Riemannian metric.
 In Section \ref{sec:Stiefel} we investigate the real symplectic Stiefel manifold, where we cover basic geometry, Riemannian and pseudo-Riemannian metrics and their geodesics as well as the Riemannian gradient. 
 Section \ref{sec:Grassmann} introduces the real symplectic Grassmann manifold as a quotient space and provides formulas for the inherited metrics and geodesics and the Riemannian gradient.
 Suitable retractions fit for replacing the actual geodesics in efficient implementations are given in Section \ref{sec:Retractions}.
 Numerical experiments are contained in Section \ref{sec:NumericalTests} and Section \ref{sec:Conc} concludes on the paper.

\section{The real symplectic group}
\label{sec:SympGroup}
Symplectic vectors spaces
are the objects of interest for the (local) study of Hamiltonian systems. An introduction can be found in \cite{Arnold2001}. By definition, a \emph{real symplectic vector space} is a real vector space $\Space{V}$ together with a nondegenerate, skew-symmetric bilinear form $\omega \colon \Space{V} \times \Space{V} \to \R$. This means $\omega$ is bilinear and fulfills
\begin{enumerate}
 \item $\omega(x,y)=0$ for all $y \in \Space{V}$ implies $x=0$ (nondegenerate),
 \item $\omega(x,y)=-\omega(y,x)$ (skew-symmetric).
\end{enumerate}
As a standard result, such a $\Space{V}$ is even-dimensional. For any subspace $\Space{U} \subset \Space{V}$, the symplectic form $\omega$ allows to the define the \emph{symplectic complement}
\begin{equation*}
    \Space{U}^\perp := \left\lbrace v \in \Space{V} \ \middle|\ \omega(v,u)= 0\ \forall u \in \Space{U} \right\rbrace.
\end{equation*}
Since in general $\Space{U}^\perp \cap \Space{U} \neq \{0\}$, four special cases of subspaces are classified. A subspace $\Space{U}$ of $(\Space{V},\omega)$ is called
\begin{multicols}{2}
\begin{enumerate}
 \item \emph{isotropic}, if $\Space{U} \subset \Space{U}^\perp$,
 \item \emph{coisotropic}, if $\Space{U}^\perp \subset \Space{U}$,
 \item \emph{Langrangian}, if $\Space{U}^\perp = \Space{U}$, and
 \item \emph{symplectic}, if $\Space{U}^\perp \cap \Space{U} = \{0\}$.
\end{enumerate}
\end{multicols}
The last case, a symplectic subspace, means that $\omega$ restricts to a symplectic form on $\Space{U}$, i.e., $(\Space{U},\omega\vert_\Space{U})$ is a symplectic space in itself.

The \emph{linear Darboux theorem} \cite{Arnold2001} states that for any two symplectic vector spaces of the same dimension, there is a linear isomorphism between them preserving the symplectic form. We can therefore restrict our considerations to the \emph{standard symplectic vector space} $(\R^{2n},\omega_0)$, where the \emph{standard symplectic form} is defined as
\begin{equation*}
    \omega_0(x,y):=x^TJ_{2n}y \quad \text{with} \quad J_{2n}:= \begin{bmatrix} 0 & I_n\\ -I_n & 0 \end{bmatrix},
\end{equation*} 
where $n \in \N$ and $I_n$ is the $n \times n$ identity matrix.
Note that $J_{2n}^T = -J_{2n}=J_{2n}^{-1}$.

The real symplectic group is the matrix Lie group of transformations which leave the standard symplectic form invariant. It has been studied for example in \cite{fiori2011, Wang2018riemannian, birtea2020optimization} with a view on applications, and in \cite{Arnold2001} from a more abstract point of view.

Define for any matrix $A \in \R^{2n \times 2k}$ the \emph{symplectic inverse} \cite{peng2016symplectic}
\begin{equation*}
    A^+:= J_{2k}^TA^TJ_{2n}.
\end{equation*} 
The \emph{real symplectic group} is then defined as
\begin{equation*}
    \Sp(2n,\R) := \left\lbrace M \in \R^{2n \times 2n} \ \middle|\ M^+ M = I_{2n}\right\rbrace
    = \left\lbrace M \in \R^{2n \times 2n} \ \middle|\ M^TJ_{2n} M = J_{2n}\right\rbrace.
\end{equation*}
For any $M \in \Sp(2n,\R)$, and any $x,y \in \R^{2n}$, it holds that
\begin{equation*}
    \omega_0(x,y) = x^T J_{2n} y = x^T M^T J_{2n} M y = \omega_0(Mx,My).
\end{equation*} 
The corresponding Lie algebra is given by the \emph{Hamiltonian matrices}
\begin{equation*}
    \sp(2n,\R) :=\left\lbrace \Omega \in \R^{2n \times 2n} \ \middle|\ \Omega^+ = -\Omega \right\rbrace.
\end{equation*}
Accordingly, the tangent space of the real symplectic group at $M$ is given by translation by $M$, i.e.
\begin{equation*}
\begin{split}
    T_M\Sp(2n,\R) &= \left\lbrace M\Omega \in \R^{2n \times 2n} \ \middle|\ \Omega \in \sp(2n,\R) \right\rbrace\\
    &= \left\lbrace \Omega M \in \R^{2n \times 2n} \ \middle|\ \Omega \in \sp(2n,\R) \right\rbrace,
\end{split}
\end{equation*}
The dimension of the symplectic group is
$
    \dim\Sp(2n,\R) = (2n+1)n,
$
see also \cite{fiori2011}.

\subsection{Pseudo-Riemannian metric}

Similarly to \cite{fiori2011}, we define a bi-invariant pseudo-Riemannian metric on $\Sp(2n,\R)$ by $h_M\colon T_M\Sp(2n,\R) \times T_M\Sp(2n,\R) \to \R$,
\begin{equation}
    \label{eq:pRmetricSp}
    h_M(X_1,X_2):=\SP{X_1,X_2}_M := \frac12\tr(X_1^+X_2),\quad X_1, X_2 \in T_M\Sp(2n,\R).
\end{equation}
The factor $\frac12$ is introduced for convenience. 
If $X_1=M\Omega_1$ and $X_2=M\Omega_2$, then
\begin{equation*}
    \SP{X_1,X_2}_M=\frac12\tr(\Omega_1^+M^+M\Omega_2) = -\frac12\tr(\Omega_1\Omega_2),
\end{equation*}
i.e. $\SP{\cdot,\cdot}_M$ is exactly $-\frac12$ times the pseudo-Riemannian metric defined in \cite{fiori2011} and therefore $-\frac12$ times the Khvedelidze–Mladenov metric \cite{Khvedelidze2002} on the general linear group.
By properties of the trace and the symplectic inverse, it can be immediately verified that the pseudo-Riemannian metric defined in this way is bi-invariant. Therefore, making use of \cite[Proposition 11.9]{ONeill1983}, the (pseudo-Riemannian) geodesics are given by the one-parameter subgroups
\begin{equation*}
    \Exp^{\Sp,h}_M(tX):=M\expm(tM^+X) = M\expm(t\Omega),
\end{equation*}
where $X=M\Omega \in T_M\Sp(2n,\R)$ and $\expm$ denotes the matrix exponential. This corresponds to \cite[Theorem 2.4]{fiori2011}.

\subsection{Riemannian metric}
The pseudo-Riemannian metric \eqref{eq:pRmetricSp} is bi-in\-va\-riant, but is not positive definite. Especially for optimization problems, a (by definition positive definite) Riemannian metric can be advantageous, and there exists a vast amount of literature concerning Riemannian optimization. While a left-invariant Riemannian metric on $\Sp(2n,\R)$ was introduced in \cite{Wang2018riemannian}, we introduce a right-invariant Riemannian metric in anticipation of the quotient structure 
that is considered in the upcoming Section~\ref{sec:Stiefel}. We also derive the corresponding gradient and geodesics.

The mapping $g_M\colon T_M\Sp(2n,\R) \times T_M\Sp(2n,\R) \to \R$,
\begin{equation}
\label{eq:RiemMetricSp}
    g_M(X_1,X_2):=\frac12\tr((X_1M^+)^TX_2M^+),\quad X_1, X_2 \in T_M\Sp(2n,\R),
\end{equation}
defines point-wise a right-invariant Riemannian metric on the real symplectic group $\Sp(2n,\R)$. The right-invariance follows from the fact that for every $N \in \Sp(2n,\R)$, $g_{MN}(X_1N,X_2N) = \frac12\tr((X_1NN^+M^+)^TX_2NN^+M^+)=g_M(X_1,X_2)$.

The Riemannian gradient for this metric is given as follows: Let $f \colon \Sp(2n,\R) \to \R$ be differentiable and let $\nabla f_M$ be the Euclidean gradient of a continuous extension of $f$ to an open subset of $\R^{2n\times 2n}$ around $M \in \Sp(2n,\R)$, evaluated at $M$. Then the Riemannian gradient of $f$ at $M$ (with respect to the metric $g_M$) is
\begin{equation*}
    \grad^g_f(M)= \nabla f_M M^TM + J_{2n} M (\nabla f_M)^T J_{2n} M \in T_M\Sp(2n,\R).
\end{equation*}
This follows from the fact that by definition $\grad^g_f(M)$ is the unique tangent vector at $M$ such that $g_M(\grad^g_f(M),X) = \D f_M(X) = \tr((\nabla f_M)^T X)$ holds for all $X \in T_M\Sp(2n,\R)$. By making use of the fact that $M^+X = -X^+ M$, $\grad^g_f(M)$ solves this equation, and $\grad^g_f(M) \in T_M(\Sp(2n,\R))$ follows from $\grad^g_f(M)M^+ = - M(\grad^g_f(M))^+$.

We can derive the Riemannian geodesics corresponding to the Riemannian metric \eqref{eq:RiemMetricSp} analogously to \cite[Proposition 4.2]{Vandereycken2013}, where the Riemannian geodesics corresponding to a right-invariant metric on the general linear group $\GL(n)$ were derived.

\begin{proposition}
    Let $M \in \Sp(2n,\R)$ and $X \in T_M\Sp(2n,\R)$. The Riemannian geodesic $\gamma$ with $\gamma(0)=M$ and $\dot\gamma(0)=X$ for the Riemannian metric \eqref{eq:RiemMetricSp} is given by
    \begin{equation*}
     \gamma(t) := \Exp^{\Sp,g}_M(tX) := \expm(t(XM^+-(XM^+)^T))\expm(t(XM^+)^T)M.
    \end{equation*} 
\end{proposition}
\begin{proof} The proof of \cite[Proposition 4.2]{Vandereycken2013} can be transferred straightforwardly to this setting.
\end{proof}

\section{The real symplectic Stiefel manifold}
\label{sec:Stiefel}

The \emph{real symplectic Stiefel manifold} is defined as
\begin{equation*}
    \SpSt(2n,2k):=\left\lbrace U \in \R^{2n \times 2k}\ \middle|\ U^+U=I_{2k} \right\rbrace 
    = \left\lbrace U \in \R^{2n \times 2k}\ \middle|\ U^TJ_{2n}U=J_{2k} \right\rbrace.
\end{equation*}
It contains the matrices $U \in \R^{2n\times 2k}$, whose column vectors form symplectic bases for the $2k$-dimensional symplectic subspaces of $(\R^{2n},\omega_0)$ and was treated in \cite{GaoSonAbsilStykel2020symplecticoptimization,GaoSonAbsilStykel2021symplecticEuclidean,SonAbsilGaoStykel2021symplecticeigenvalue}.
Note the formal similarity with the (compact) Stiefel manifold $\text{St}(n,k) = \left\lbrace U \in \R^{n \times k}\ \middle|\ U^TU=I_{k} \right\rbrace$.
As a novelty, and in contrast to the aforementioned references, we will pursue a Lie group-based approach to study the real symplectic Stiefel manifold. We will furthermore introduce a new pseudo-Riemannian and a new Riemannian metric and derive the geodesics for both.

Denote the projection onto the first $k$ columns of a matrix, when multiplied from the right, by
\begin{equation}
\label{eq:I_nk}
    I_{n,k}:= \begin{bmatrix} I_k \\ 0 \end{bmatrix} \in \R^{n \times k},
\end{equation}
and the projection onto the first $k$ and the $(n+1)$th to the $(n+k)$th column by
\begin{equation*}
    E:=\begin{bmatrix} I_{n,k} & 0 \\ 0 & I_{n,k} \end{bmatrix} \in \R^{2n \times 2k}.
\end{equation*}
Our first goal is to recognize the real symplectic Stiefel manifold as 
a quotient of the real symplectic group.
To this end, we introduce the following canonical projection:
\begin{equation}
\label{eq:projStiefel}
    \pi\colon \Sp(2n,\R) \to \SpSt(2n,2k),\ M \mapsto ME.
\end{equation}

\begin{proposition}
The real symplectic Stiefel manifold is diffeomorphic to the quotient
\begin{equation*}
    \SpSt(2n,2k)\cong \Sp(2n,\R)/\Sp(2(n-k),\R).
\end{equation*} 
It has dimension $\dim(\SpSt(2n,2k)) = (4n-2k+1)k$.
\end{proposition}
\begin{proof}
The set $\SpSt(2n,2k)$ is the orbit of $E$ under the group action of $\Sp(2n,\R)$ that is induced by left-multiplication. The stabilizer 
\begin{equation*}
    \stab_{E}:=\left\lbrace M \in \Sp(2n,\R) \ \middle|\ ME=E \right\rbrace
\end{equation*}
of this group action is isomorphic to $\Sp(2(n-k),\R)$. From \cite[Theorem 21.20]{Lee2012smooth}, it follows that $\SpSt(2n,2k)$ has a unique smooth manifold structure for which the group action is smooth. It furthermore follows that the dimension of the real symplectic Stiefel manifold is
\begin{equation*}
    \dim(\SpSt(2n,2k)) = \dim(\Sp(2n,\R))-\dim(\Sp(2(n-k),\R))= (4n-2k+1)k,
\end{equation*}
in accordance with \cite{GaoSonAbsilStykel2020symplecticoptimization}. From \cite[Theorem 21.18]{Lee2012smooth}, the existence of a diffeomorphism between $\SpSt(2n,2k)$ and $\Sp(2n,\R)/\Sp(2(n-k),\R)$ follows.
\end{proof}
 The quotient manifold structure of the real symplectic Stiefel manifold 
 (and of the real symplectic Grassmann manifold, which is to be discussed later on) with the symplectic group 
 as the associated total space is visualized in Figure~\ref{fig:QuotientStructure}.

\begin{figure}
\begin{center}
 \includegraphics{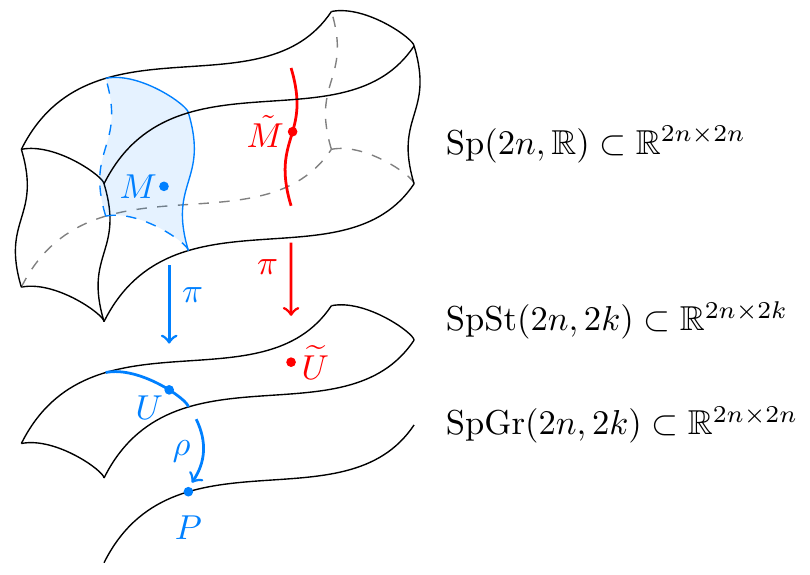}
\end{center}
\caption{Visualization of the quotient structure of the real symplectic Grassmann and Stiefel manifold with respect to the symplectic group. Any point $P \in \SpGr(2n,2k)$ has an equivalence class in $\SpSt(2n,2k)$ as its pre-image under $\rho$, visualized by the blue line through $U \in \SpSt(2n,2k)$. This equivalence class in turn has again an equivalence class of symplectic matrices in $\Sp(2n,\R)$ as its pre-image under $\pi$, visualized by the blue area around $M$. The equivalence class of a single point $\tilde U \in \SpSt(2n,2k)$ is of lower dimension, visualized by the red line through $\tilde M \in \Sp(2n,\R)$.}
\label{fig:QuotientStructure}
\end{figure}

The Lie group approach allows to represent tangent vectors in a similar way as is common for the standard Stiefel manifold $\St(n,k)$.
As the projection $\pi$ in \eqref{eq:projStiefel} is surjective, for every $U \in \SpSt(2n,2k)$, there is an $M \in \Sp(2n,\R)$ such that $U = ME$. 
Define a \emph{symplectic complement} of $U$ by $U^s:= ME^s$, where
\begin{equation*}
    E^s := \begin{bmatrix}
            0_{k\times(n-k)} & 0_{k\times(n-k)}\\
            I_{n-k} & 0_{n-k}\\
            0_{k\times(n-k)} & 0_{k\times(n-k)}\\
            0_{n-k} & I_{n-k}
           \end{bmatrix} \in \R^{2n \times 2(n-k)},
\end{equation*}
i.e. the projection onto the columns complementary to those selected by $E$. Note that $E^+=E^T$ and $(U^s)^+U^s = (E^s)^+M^+ME^s = I_{2(n-k)}$, i.e. $U^s \in \SpSt(2n,2(n-k))$. Furthermore  $U^+U^s = E^+M^+ME^s=0$ and $I_{2n}-UU^+ = M(I_{2n}-EE^+)M^+=U^s(U^s)^+$.

\begin{proposition}
    The tangent space at $U \in \SpSt(2n,2k)$ is given by
    \begin{equation}
    \label{eq:TangentSpaceStiefel}
    \begin{split}
        T_U\SpSt(2n,2k) &= \left\lbrace UA+U^sB \in \R^{2n\times 2k} \ \middle|\ A \in \sp(2k,\R), B \in \R^{2(n-k)\times 2k}\right\rbrace\\
        &=\left\lbrace \Delta \in \R^{2n\times 2k}\ \middle|\ U^+\Delta \in \sp(2k,\R)\right\rbrace.
    \end{split}
    \end{equation}
\end{proposition}
    \begin{proof}
     For every tangent vector $\Delta \in T_U\SpSt(2n,2k)$, there is a curve $\gamma \colon (-\varepsilon,\varepsilon) \to \SpSt(2n,2k)$, with $\gamma(0)=U$ and $\dot\gamma(0)=\Delta$. Since $\gamma(t)^+\gamma(t)=I_{2k}$, differentiating and evaluating at $t=0$ leads to
     $
        U^+\Delta = -\Delta^+U.
     $\\
     Therefore $(U^+\Delta)^+ = \Delta^+U=-U^+\Delta$, and $A := U^+\Delta \in \sp(2k,\R)$. Furthermore,
     \begin{equation*}
        \Delta = UU^+\Delta + (I_{2n}-UU^+)\Delta = UA + U^s(U^s)^+\Delta.
     \end{equation*}
     Counting dimensions and defining $B := (U^s)^+\Delta$ yields the result.
    \end{proof}
Note that the tangent space parametrization \cite[(3.8b)]{GaoSonAbsilStykel2020symplecticoptimization} is similar to \eqref{eq:TangentSpaceStiefel}, but the chosen complement there is not necessarily a symplectic complement.

\subsection{Pseudo-Riemannian metric  on \texorpdfstring{$\SpSt(2n,2k)$}{SpSt(2n,2k)}}

According to our quotient Lie group approach, $\SpSt(2n,2k)$ inherits a pseudo-Riemannian metric from the pseudo-Riemannian metric \eqref{eq:pRmetricSp} on 
the total space $\Sp(2n,\R)$  in a natural way, by making use of horizontal lifts. 
A big advantage of this construction is that the corresponding geodesics can then be obtained via the projection of {\em horizontal geodesics}, i.e., geodesics with horizontal tangent vectors on $\Sp(2n,\R)$ \cite[Corollary 7.46]{ONeill1983}.

Splitting the Lie algebra $\sp(2n,\R)$ into a vertical and horizontal part with respect to the projection \eqref{eq:projStiefel} and the pseudo-Riemannian metric $h$ from \eqref{eq:pRmetricSp} gives
\begin{equation}
\label{eq:VerHorSt}
    \sp(2n,\R) = \Ver^\pi\sp(2n,\R) \oplus \Hor^{\pi,h}\sp(2n,\R),
\end{equation}
with {\em vertical space}
\begin{equation*}
    \Ver^\pi\sp(2n,\R):= \ker \D\pi_{E}=\left\lbrace
    \begin{bsmallmatrix}
    0 & 0 & 0 & 0\\
    0 & A & 0 & B\\
    0 & 0 & 0 & 0\\
    0 & C & 0 & -A^T
    \end{bsmallmatrix}
    \ \middle|\ 
    \begin{matrix}
        A \in \R^{(n-k)\times (n-k)}\\
        B,C \in \Sym_{n-k}
    \end{matrix}
    \right\rbrace,
\end{equation*}
and {\em horizontal space}
\begin{equation*}
\begin{split}
    \Hor^{\pi,h}\sp(2n,\R):=&(\Ver^\pi\sp(2n,\R))^{\perp,h} \subset \sp(2n,\R)\\
    =&\left\lbrace
    \begin{bsmallmatrix}
    A_1   & A_2^T   & B_1    & B_2^T    \\
    A_3   & 0     & B_2    & 0      \\
    C_1   & C_2^T & -A_1^T & -A_3^T \\
    C_2   & 0     & -A_2 & 0
    \end{bsmallmatrix}
     \ \middle|\ 
    \begin{matrix}
        A_1 \in \R^{k \times k}\\
        B_1,C_1 \in \Sym_k,\\
        A_2,A_3,B_2,C_2 \in \R^{(n-k) \times k}
    \end{matrix}
    \right\rbrace,
\end{split}
\end{equation*}
where the orthogonal complement is taken with respect to the pseudo-Riemannian metric $h$ of \eqref{eq:pRmetricSp}.
\begin{proposition}
\label{prop:HorOmegaStProperty}
    For any $\Omega \in \sp(2n,\R)$, it holds that $\Omega \in \Hor^{\pi,h}\sp(2n,\R)$ if and only if
    \begin{equation}
    \label{eq:HorOmegaStProperty}
    \begin{split}
        \Omega &= \left(I_{2n}-\frac12 EE^+\right)\Omega EE^+ - EE^+ \Omega^+ \left(I_{2n}-\frac12 EE^+\right)\\
        &=\Omega EE^+ + EE^+ \Omega - EE^+\Omega EE^+.
    \end{split}
    \end{equation}
\end{proposition}
    \begin{proof}
        Follows by a straightforward calculation.
    \end{proof}
    Eventually, we will exploit abstract results of semi-Riemannian geometry \cite[\S 11]{ONeill1983} for determining the geo\-de\-sics.
    To enable this, we show next that $\SpSt(2n,2k)\cong \Sp(2n,\R)/\stab_E$ is 
    \emph{reductive}, and \emph{naturally reductive} with respect to the pseudo-Riemannian metric $h$ of \eqref{eq:pRmetricSp},
    see \cite[Definition 11.21 \& 11.23]{ONeill1983} for an explanation of these terms.
    \begin{lemma}
    \label{lemma:SpSt_naturally_reductive}
      The real symplectic Stiefel manifold $\SpSt(2n,2k)$ is \emph{reductive}.
      With respect to the pseudo-Riemannian metric $h$, it is \emph{naturally reductive}.
    \end{lemma}
    \begin{proof}
        In view of \eqref{eq:VerHorSt} and the fact that $\Ver^\pi\sp(2n,\R)$ is isomorphic to the Lie algebra of $\Sp(2(n-k),\R) \cong \stab_E$, 
        we need to show that the complementary subspace $\Hor^{\pi,h}\sp(2n,\R)$ is $\Ad(\stab_E)$-invariant in order to establish reductiveness,
        where $\Ad_M(\Omega) = M \Omega M^+$, see \cite[\S 11, p. 303]{ONeill1983}.
        %
        In fact, for every $M \in \stab_E$ (and therefore $M^+ \in \stab_E$) and every $\Omega \in \Hor^{\pi,h}\sp(2n,\R)$, it holds that
        \begin{align*}
           \Ad_M(\Omega) = M\Omega M^+ &= M \Omega EE^+ M^+ + MEE^+\Omega M^+ - MEE^+ \Omega EE^+ M^+\\
            &= M\Omega M^+ EE^+ + EE^+ M\Omega M^+ - EE^+ M\Omega M^+ EE^+.
        \end{align*}
        By Proposition~\eqref{prop:HorOmegaStProperty}, it follows that $\Ad_M(\Omega) \in \Hor^{\pi,h}\sp(2n,\R)$, which means that $\Hor^{\pi,h}\sp(2n,\R)$ is $\Ad(\stab_E)$-invariant. Therefore, $\SpSt(2n,2k)$ is reductive.
        
        The fact that $\SpSt(2n,2k)$ is naturally reductive with respect to $h$ follows from a direct calculation, by making use of the fact that the projection from $\sp(2n,\R)$ onto $\Hor^{\pi,h}\sp(2n,\R)$ is given by $\Omega \mapsto \Omega - (I_{2n}-EE^+)\Omega(I_{2n}-EE^+)$.
    \end{proof}

Recall that the Lie algebra  $\sp(2n,\R) = T_I\Sp(2n,\R)$ is the tangent space at the identity.
By left translation, every tangent space $T_M\Sp(2n,\R)$ can be split into a vertical and horizontal part,
\begin{equation}
\begin{split}
\label{eq:StiefelVerHorPseudoRiem}
    \Ver^\pi_M\Sp(2n,\R) &:= \left\lbrace M\Omega \ \middle|\  \Omega \in  \Ver^\pi\sp(2n,\R) \right\rbrace\\
    \Hor^{\pi,h}_M\Sp(2n,\R) &:= \left\lbrace M\Omega \ \middle|\  \Omega \in  \Hor^{\pi,h}\sp(2n,\R) \right\rbrace
\end{split}.
\end{equation}
The horizontal space at $M \in \Sp(2n,\R)$, i.e., $\Hor^{\pi,h}_M\Sp(2n,\R)$, is isomorphic to the tangent space $T_{\pi(M)}\SpSt(2n,2k)$. For $M\Omega \in \Hor^{\pi,h}_M\Sp(2n,\R)$, it holds that \begin{align*}
    \D \pi_M(M\Omega)= M\Omega E = M \begin{bsmallmatrix}
    A_1   & A_2^T   & B_1    & B_2^T    \\
    A_3   & 0     & B_2    & 0      \\
    C_1   & C_2^T & -A_1^T & -A_3^T \\
    C_2   & 0     & -A_2 & 0
    \end{bsmallmatrix} E = M \begin{bsmallmatrix}
    A_1   & B_1 \\
    A_3   & B_2 \\
    C_1   & -A_1^T \\
    C_2   & -A_2
    \end{bsmallmatrix} = UA + U^s B,
\end{align*}
where $A = \begin{bsmallmatrix} A_1 & B_1 \\ C_1 & -A_1^T\end{bsmallmatrix}$, $B=\begin{bsmallmatrix}A_3 & B_2 \\ C_2 & -A_2 \end{bsmallmatrix}$ and $U = ME$. It follows therefore that the application of $\D \pi_M$ to $\Hor^{\pi,h}_M\Sp(2n,\R)$ gives \eqref{eq:TangentSpaceStiefel}. This implies that $T_U\SpSt(2n,2k)$ can also be parameterized as
\begin{equation*}
\label{eq:TangentSpaceStiefel_Large}
    T_U\SpSt(2n,2k) = \left\lbrace M\Omega E \ \middle|\  M \in \pi^{-1}(U),\ \Omega \in \Hor^{\pi,h}\sp(2n,\R) \right\rbrace.
\end{equation*}
\begin{proposition}[Alternative tangent vector parameterization]
    Every tangent vector $\Delta \in T_U\SpSt(2n,2k)$ is of the form $\Delta = \tilde\Omega U$, where 
    \begin{equation*}
        \tilde\Omega = M\Omega M^+ \in \sp(2n,\R),
    \end{equation*} 
    with $\Omega \in \Hor^{\pi,h}\sp(2n,\R)$, is unique. 
    It can be calculated via
\begin{equation}
\label{eq:OmegatildefromDelta}
    \tilde\Omega(U,\Delta) = \left(I_{2n}-\frac12UU^+\right)\Delta U^+ - U\Delta^+\left(I_{2n} -\frac12 UU^+\right).
\end{equation}
\end{proposition}
\begin{proof}
 This can be seen by making use of $U=ME$ and Proposition \ref{prop:HorOmegaStProperty}.
\end{proof}
Equation \eqref{eq:OmegatildefromDelta} corresponds to $S_{X,Y}J$ from \cite[Proposition 4.3]{GaoSonAbsilStykel2020symplecticoptimization}.

Via horizontal lifts, a pseudo-Riemannian metric on the real symplectic Stiefel manifold can be defined as follows: For two tangent vectors $\Delta_1,\Delta_2 \in T_U\SpSt(2n,2k)$ and $U=ME$, calculate $\tilde{\Omega}(U,\Delta_i)$ according to \eqref{eq:OmegatildefromDelta}, $i=1,2$. The horizontal lift to $\Hor^{\pi,h}_M\Sp(2n,\R)$ is then given by
\begin{equation*}
(\Delta_i)^\hor_M = M\Omega_i = \tilde{\Omega}(U,\Delta_i) M,                                                                                                                                                                                                                                                                                                                                                                          \end{equation*}
where $\Omega_i = M^+\tilde{\Omega}(U,\Delta_i) M$. This follows from \[\D\pi_M((\Delta_i)^\hor_M)=(\Delta_i)^\hor_M E= \tilde{\Omega}(U,\Delta_i)ME = \tilde{\Omega}(U,\Delta_i) U = \Delta_i,\] and the fact that $\Omega_i$ fulfills Proposition~\ref{prop:HorOmegaStProperty}.
In the following, 
we exploit that pseudo-Riemannian submersions
\cite[Definition 7.44]{ONeill1983} are particularly useful
to find the ge\-o\-de\-sics in a quotient space, given that the geodesics in the associated total space are known,
see \cite[Corollary 7.46]{ONeill1983}.
\begin{proposition}
\label{prop:pRmetricSt}
    Let $\Delta_i = U A_i + U^s B_i \in T_U\SpSt(2n,2k)$, $i=1,2$.
    A pseudo-Riemannian metric is defined by $h^\SpSt_U\colon T_U\SpSt(2n,2k) \times T_U\SpSt(2n,2k) \to \R$,
    \begin{equation}
    \label{eq:pRmetricSt}
    \begin{split}
        h^\SpSt_U(\Delta_1,\Delta_2) &:= \SP{\Delta_1,\Delta_2}_U := \SP{(\Delta_2)^\hor_M,(\Delta_2)^\hor_M}_M\\
        &= \tr\left(\Delta_1^+\left(I_{2n}-\frac12 UU^+\right)\Delta_2\right) = \frac12\tr(A_1^+A_2) + \tr(B_1^+B_2).
    \end{split}
    \end{equation} 
    With respect to the metric $h^\SpSt_U$, $\pi$ is a pseudo-Riemannian submersion.
\end{proposition}
    \begin{proof}
        A direct calculation shows the identities of \eqref{eq:pRmetricSt}. Since $h$ is a pseudo-Riemannian metric, $h_U^\SpSt$ is a pseudo-Riemannian metric as well. The projection $\pi$ is then a pseudo-Riemannian submersion by Lemma~\ref{lemma:SpSt_naturally_reductive} and \cite[Lemma 11.24]{ONeill1983}.
    \end{proof}


The geodesics, calculated via the exponential mapping with respect to $h_U^\SpSt$, can be found via the projection of the exponential mapping in $\Sp(2n,\R)$.
\begin{proposition}
    Let $U \in \SpSt(2n,2k)$ and $M \in \pi^{-1}(U) \subset \Sp(2n,\R)$. Furthermore, let $\Delta = M\Omega \in T_U\SpSt(2n,2k)$. The geodesic $\gamma$ with respect to the pseudo-Riemannian metric \eqref{eq:pRmetricSt} that  starts from $\gamma(0)= U$ in direction $\dot\gamma(0)= \Delta$ is
    \begin{equation}\label{eq:StiefelBigExp}
    \begin{split}
        \gamma(t) &:= \Exp^{\SpSt,h}_U(t\Delta) := \pi(\Exp^{\Sp,h}_M(t\Delta^\hor_M))\\
        &= M \expm(t\Omega)E = \expm(t\tilde\Omega(U,\Delta))U,
    \end{split}
    \end{equation}
    with $\tilde\Omega(U,\Delta) = M\Omega M^+$ from \eqref{eq:OmegatildefromDelta}.
\end{proposition}
    \begin{proof}
        By \cite[Corollary 7.46]{ONeill1983}, horizontal geodesics in $\Sp(2n,\R)$ are mapped to geodesics in the quotient $\SpSt(2n,2k)$ under the pseudo-Riemannian submersion $\pi$. The facts that $\gamma(0)=U$ and $\dot\gamma(0)=M\Omega=\Delta$ are immediate.
    \end{proof}

In the form of \eqref{eq:StiefelBigExp}, the exponential mapping depends on the matrix exponential of a $2n\times 2n$ matrix. For tangent vectors $\Delta \in T_U\SpSt(2n,2k)$ with $\Delta^+(I_{2n}-UU^+)\Delta$ invertible, we can reduce the computational complexity to $4k \times 4k$. This is rendered possible by the fact that for $X,Y \in \R^{n \times k}$ with $Y^TX \in \R^{k\times k}$ non-singular, we have from \cite[Prop. 3]{Celledoni2000} that
\begin{equation}\label{eq:reducedexp}
    \expm(XY^T) = I_n + X(\expm(Y^TX)- I_k)(Y^TX)^{-1}Y^T.
\end{equation}
%
\begin{proposition}\label{prop:StiefelReducedExp}
    Let $U \in \SpSt(2n,2k)$ and $\Delta \in T_U\SpSt(2n,2k)$. Define $A = U^+\Delta$ and $H =\Delta - UA$. If $H^+H$ is invertible,  then the geodesic from $U$ in direction $\Delta$ is given by
    \begin{eqnarray}
       \label{eq:StiefelReducedExp1} 
       \Exp^{\SpSt,h}_U(t\Delta)&=& \begin{bmatrix} U & \frac12 U A + H \end{bmatrix}
       \expm\left(t\begin{bmatrix} \frac12 A & \frac14 A^2 - H^+H\\  
                                     I_{2k}  & \frac12 A 
                   \end{bmatrix}\right)
                   \begin{bmatrix} I_{2k}\\0\end{bmatrix}.
    \end{eqnarray} 
\end{proposition}
%
\begin{proof}
    We start from \eqref{eq:StiefelBigExp} with $\tilde\Omega(U,\Delta)$ according to \eqref{eq:OmegatildefromDelta}.
    Introducing
    \begin{equation*}
     X= \begin{bmatrix} (I_{2n}-\frac12 UU^+)\Delta & -U \end{bmatrix} = \begin{bmatrix} \frac12 UA + H & -U\end{bmatrix} \in \R^{2n \times 4k}
    \end{equation*} and 
    $Y^T = \begin{bmatrix} U^+ \\ \Delta^+(I_{2n}-\frac12 UU^+) \end{bmatrix} \in \R^{4k \times 2n}$, we have $\tilde\Omega(U,\Delta) = XY^T$.
    Furthermore
    \begin{align*}
        Y^TX = \begin{bmatrix} \frac12 U^+\Delta & -I_{2k}\\  \Delta^+(I_{2n}-\frac34UU^+)\Delta & -\frac12 \Delta^+ U \end{bmatrix}
        = \begin{bmatrix} \frac12 A & -I_{2k}\\ H^+H -\frac14A^2 & \frac12 A \end{bmatrix}.
    \end{align*}
    If $(Y^TX)^{-1}$ exists, it is given by
    \begin{align*}
        (Y^TX)^{-1}= \begin{bmatrix} \frac12 (H^+H)^{-1}A & (H^+H)^{-1}\\ \frac14 A(H^+H)^{-1}A - I_{2k} & \frac12 A(H^+H)^{-1}\end{bmatrix},
    \end{align*}
    and therefore $Y^TX$ is invertible if and only if $H^+H$ is invertible. It furthermore holds that $(Y^TX)^{-1} Y^T U = (Y^TX)^{-1}Y^TX \begin{bsmallmatrix} 0 \\ -I_{2k}\end{bsmallmatrix} =  \begin{bsmallmatrix} 0 \\ -I_{2k}\end{bsmallmatrix}$. 
    By \eqref{eq:StiefelBigExp} it holds that $\Exp^{\SpSt,h}_U(t\Delta) = \expm(t\tilde\Omega)U= \expm(tXY^T)U$. Now by \eqref{eq:reducedexp}, it holds for $t\neq 0$ that
    \begin{align*}
        \expm(tXY^T)U &= (I_{2n} + tX(\expm(tY^TX)-I_{2k})(tY^TX)^{-1}Y^T)U\\
        &= U + X(\expm(tY^TX)-I_{2k})(Y^TX)^{-1}Y^TU\\
        &= U + X\expm(tY^TX)\begin{bmatrix}0\\-I_{2k}\end{bmatrix} - X\begin{bmatrix}0\\-I_{2k}\end{bmatrix}\\
        &= X\expm(tY^TX)\begin{bmatrix}0\\-I_{2k}\end{bmatrix}.
    \end{align*}
    Moreover, $\lim_{t\to 0} X\expm(tY^TX)\begin{bsmallmatrix}0\\-I_{2k}\end{bsmallmatrix} = U = \Exp^{\SpSt,h}_U(0\cdot\Delta)$.\\
    The form \eqref{eq:StiefelReducedExp1} is obtained by 
    \[
     X\expm(tY^TX)\begin{bmatrix}0\\-I_{2k}\end{bmatrix}
     =XJ_{4k}^T\expm(tJ_{4k}Y^TXJ_{4k}^T)J_{4k}\begin{bmatrix}0\\-I_{2k}\end{bmatrix}.
    \]
\end{proof}
%
Note that for the calculation of \eqref{eq:StiefelReducedExp1} we don't need the invertibility of $H^+H$, and one can check that the right hand side is always an element in the symplectic Stiefel manifold $\SpSt(2n,2k)$. We can therefore always apply \eqref{eq:StiefelReducedExp1} to calculate a curve.

\begin{remark}
    The simplified formula for the symplectic Stiefel exponential \eqref{eq:StiefelReducedExp1} is similar to the so-called quasigeodesic retraction defined in \cite[Lemma 5.1]{GaoSonAbsilStykel2020symplecticoptimization}, which in our notation is given as
    \begin{equation}
    \label{eq:QuasiGeodesicRetraction}
        \Ret^\qgeo_U(\Delta) = \begin{bmatrix} U & \Delta\end{bmatrix}\expm\left(\begin{bmatrix}U^+\Delta & -\Delta^+\Delta\\ I_{2k} & U^+\Delta \end{bmatrix}\right)\begin{bmatrix} I_{2k}\\ 0\end{bmatrix}\expm(-U^+\Delta).
    \end{equation}
    The two curves are however not identical. Note also the structural similarity with the formula for the Euclidean Stiefel geodesics of \cite[Section 2.2.2]{EdelmanAriasSmith1999}.
\end{remark}

\subsection{Right-invariant Riemannian metric  on \texorpdfstring{$\SpSt(2n,2k)$}{SpSt(2n,2k)}}
The real symplectic Stiefel manifold may be equipped with different Riemannian metrics.
The so-called ca\-no\-nical-like metric has been studied in \cite{GaoSonAbsilStykel2020symplecticoptimization}, while \cite{GaoSonAbsilStykel2021symplecticEuclidean} considers a restriction of the Euclidean metric. To the best of the authors' knowledge, the geodesics for these metrics are unknown. 
Complementary to the aforementioned Riemannian metrics, 
we use the Riemannian metric $g_M$ of \eqref{eq:RiemMetricSp} on $\Sp(2n,\R)$ to introduce a third Riemannian metric on $\SpSt(2n,2k)$ via a horizontal lift, which allows us to find the corresponding geodesics.
This metric is invariant under the group action of $\Sp(2k,\R)$ from the right and induces therefore a Riemannian metric on the symplectic subspaces that will be considered in Subsection~\ref{subsec:SymplGrassRiemannianMetric}.

We begin by splitting the tangent space $T_M\Sp(2n,\R)$ at $M \in \Sp(2n,\R)$ into a vertical and a horizontal part with respect to $g_M$ and the projection $\pi$ from \eqref{eq:projStiefel},
\begin{equation*}
    T_M\Sp(2n,\R) = \Ver^\pi_M\Sp(2n,\R) \oplus \Hor^{\pi,g}_M\Sp(2n,\R).
\end{equation*} 
 As the vertical part is defined as the kernel of $\D \pi_M$, it is the same as in \eqref{eq:StiefelVerHorPseudoRiem}. 
 The horizontal part, however, is different, since it is now given as the orthogonal complement of $\Ver^\pi_M\Sp(2n,\R)$ with respect to the metric $g_M$ of \eqref{eq:RiemMetricSp}. This yields
\begin{equation}
\label{eq:hor_pi_gM}
    \Hor^{\pi,g}_M\Sp(2n,\R) = \left\lbrace \bar\Omega M \ \middle|\ \bar\Omega = \bar\Omega P + P \bar\Omega - P\bar\Omega P \in \sp(2n,\R)\right\rbrace.
\end{equation}
Here, $P = J_{2n}^T U U^+ J_{2n}$ and $U = \pi(M) = ME$. 
Equation \eqref{eq:hor_pi_gM} can be established as follows: Any horizontal tangent vector $X = \bar\Omega M$, with $\bar\Omega \in \sp(2n,\R)$, fulfills for all $Y = M \Omega \in \Ver^\pi_M\Sp(2n,\R)$
\begin{align*}
    0 = g_M(X,Y) = \frac12 \tr(\bar\Omega^TM\Omega M^+) = \frac12\tr(M^+ J_{2n}\bar\Omega J_{2n}M\Omega).
\end{align*}
Define $\hat\Omega := M^+ J_{2n} \bar\Omega J_{2n}M$. Then $0 = \frac12\tr(\hat\Omega \Omega)$ for all $\Omega \in \Ver^\pi\sp(2n,\R)$ implies that $\hat\Omega$ fulfills \eqref{eq:HorOmegaStProperty}.
By making use of $E = M^+ME = M^+U$, it follows that
\begin{align*}
    \bar\Omega &= J_{2n}M\hat\Omega M^+J_{2n}\\
    &= J_{2n}M(\hat\Omega EE^+ +EE^+\hat\Omega - EE^+\hat\Omega EE^+) M^+J_{2n}\\
    &= \bar\Omega J_{2n}UU^+J_{2n}^T + J_{2n}UU^+J_{2n}^T\bar\Omega - J_{2n}UU^+J_{2n}^T\bar\Omega J_{2n}UU^+J_{2n}^T.
\end{align*}
Conversely, if $\bar\Omega$ fulfills the above equation, then it follows that $\hat\Omega := M^+ J_{2n} \bar\Omega J_{2n}M$ fulfills \eqref{eq:HorOmegaStProperty} and therefore $\bar\Omega M \in \Hor^{\pi,g}_M\Sp(2n,\R)$.

As usual, for $U=\pi(M)$ we can identify the tangent space $T_U\SpSt(2n,2k)$ with the horizontal space $\Hor^{\pi,g}_M\Sp(2n,\R)$. Any $\Delta \in T_U\SpSt(2n,2k)$ is of the form $\Delta = \pi(\bar\Omega M)= \bar\Omega U$ for some $\bar\Omega M \in \Hor^{\pi,g}_M\Sp(2n,\R)$, and we can find the horizontal lift 
\begin{equation}
\label{eq:RiemHorLiftStiefel}
    \Delta^{\hor,g}_M=\bar\Omega(\Delta) M
\end{equation} 
via
\begin{equation}
 \label{eq:barOmegafromDelta}
 \bar\Omega(\Delta)= \Delta (U^TU)^{-1} U^T + J_{2n}U(U^TU)^{-1}\Delta^T(I_{2n}-J_{2n}^TU(U^TU)^{-1}U^TJ_{2n})J_{2n}.
\end{equation}
This follows from the facts that
\begin{enumerate}
 \item $\bar\Omega(\Delta)^+ = -\bar\Omega(\Delta)$, so $\bar\Omega(\Delta) \in \sp(2n,\R)$,
 \item $\bar\Omega(\Delta) U = \Delta$, and
 \item $\bar\Omega(\Delta) = \bar\Omega(\Delta) P + P\bar\Omega(\Delta) - P\bar\Omega(\Delta) P$, with $P = J_{2n}UU^+J_{2n}^T$,
\end{enumerate}
where the last equation follows from a straightforward calculation.

The right action of $\Sp(2(n-k),\R) \cong \stab_E$ on $\Sp(2n,\R)$ is vertical, i.e., $\pi(MN)=\pi(M)$ for all $N \in \stab_E$. The action is also transitive on fibers, i.e., for $M,M' \in \Sp(2n,\R)$ with $\pi(M)= U = \pi(M')$ it holds that $M M^+M' = M'$ and $M^+M' \in \stab_E$. It is furthermore isometric, by right-invariance of \eqref{eq:RiemMetricSp}. From \cite[Theorem 2.28]{Lee2018riemannian}, it follows that there is a unique Riemannian metric on $\SpSt(2n,2k)$ such that $\pi$ is a Riemannian submersion.
This Riemannian metric is given via the horizontal lift.
\begin{proposition}
The Riemannian metric on $\SpSt(2n,2k)$, for which $\pi$ is a Riemannian submersion, is right-invariant and given point-wise by
\begin{equation}
\label{eq:RiemMetricStiefel}
 \begin{split}
 g^\SpSt_U &\colon T_U\SpSt(2n,2k) \times T_U\SpSt(2n,2k) \to \R,\\
    g^\SpSt_U(\Delta_1,\Delta_2) &:= g_M((\Delta_1)^{\hor,g}_M,(\Delta_2)^{\hor,g}_M)\\
    &= \tr\left(\Delta_1^T\left(I_{2n}-\frac12 J_{2n}^TU(U^TU)^{-1}U^TJ_{2n}\right)\Delta_2(U^TU)^{-1}\right).
 \end{split}
\end{equation}
\end{proposition}
\begin{proof}
 The Riemannian submersion property and right-invariance hold by the definition of $g^\SpSt_U$ via the horizontal lift. The second equality follows from the combination of \eqref{eq:RiemMetricSp}, \eqref{eq:RiemHorLiftStiefel} and \eqref{eq:barOmegafromDelta}.
\end{proof}

The Riemannian gradient of a function $f\colon \SpSt(2n,2k) \to \R$ with respect to $g^\SpSt$ is given by
\begin{equation}
    \label{eq:RiemGradientStiefel}
    \grad^g_f(U)=\nabla f(U) U^TU + J_{2n}U(\nabla f(U))^T J_{2n}U,
\end{equation}
where $\nabla f(U)$ denotes the Euclidean gradient of a smooth extension of $f$ around $U \in \SpSt(2n,2k)$ in $\R^{2n \times 2k}$ at $U$. This holds because $U^+\grad^g_f(U)=-(\grad^g_f(U))^+U$, which implies $\grad^g_f(U) \in T_U\SpSt(2n,2k)$, and because $\grad^g_f(U)$ solves
\begin{equation*}
    g^\SpSt_U(\grad^g_f(U),\Delta) = \D f_U(\Delta) = \tr((\nabla f(U))^T\Delta)
\end{equation*}
for all $\Delta \in T_U\SpSt(2n,2k)$.

Riemannian geodesics on $\Sp(2n,\R)$ with a horizontal tangent vector at every point project to Riemannian geodesics on $\SpSt(2n,2k)$ by \cite[Corollary 7.46]{ONeill1983}.
(Mind that the referenced result is stated in the pseudo-Riemannian setting, but also holds true in the Riemannian case.) 
We show that Riemannian geodesics on $\Sp(2n,\R)$ with initial horizontal tangent vector have a horizontal tangent vector throughout.
\begin{lemma}
\label{lemma:HorStiefelGeodesic}
    Let $M \in \Sp(2n,\R)$ and $X \in \Hor^{\pi,g}_M\Sp(2n,\R)$. Define $\gamma(t):=\Exp^{\Sp,g}_M(tX)$. Then $\dot\gamma(t) \in \Hor^{\pi,g}_{\gamma(t)}\Sp(2n,\R)$.
\end{lemma}
\begin{proof}
 Let $U = \pi(M)$ and $U(t) = \pi(\gamma(t))=\gamma(t)E$. Furthermore, let $P(t) := J_{2n}U(t)U(t)^+J_{2n}^T$. By the structure of the horizontal space, it holds $X = \bar\Omega M$. Define $x(t) := \dot\gamma(t)\gamma(t) \in \sp(2n,\R)$. Then, by \eqref{eq:hor_pi_gM}, $\dot\gamma(t) \in \Hor^{\pi,g}_{\gamma(t)}\Sp(2n,\R)$ is equivalent to
 \begin{equation*}
    x(t) = P(t)x(t) + x(t)P(t) - P(t)x(t)P(t).
 \end{equation*}
 With
 \begin{itemize}
  \item $x(t) = \expm(t(\bar\Omega-\bar\Omega^T))\bar\Omega\expm(-t(\bar\Omega-\bar\Omega^T))$,
  \item $U(t) = \expm(t(\bar\Omega-\bar\Omega^T))\expm(t\bar\Omega^T)U$,
  \item $J_{2n}^T\expm(t(\bar\Omega-\bar\Omega^T))J_{2n} = \expm(t(\bar\Omega-\bar\Omega^T))$, \item $J_{2n}\expm(t\bar\Omega)J_{2n}= \expm(-t\bar\Omega)$ and
  \item $P(t)= \expm(t(\bar\Omega-\bar\Omega^T))\expm(-t\bar\Omega)P(0)\expm(t\bar\Omega)\expm(-t(\bar\Omega-\bar\Omega^T))$,
 \end{itemize}
 the claim follows by a straightforward calculation.
\end{proof}

We are now ready to state the Riemannian geodesics on $\SpSt(2n,2k)$ with respect to the Riemannian metric $g^\SpSt$ from
\eqref{eq:RiemMetricStiefel}.

\begin{proposition}
    Let $U \in \SpSt(2n,2k)$ and $\Delta \in T_U\SpSt(2n,2k)$. Let $M \in \pi^{-1}(U) \subset \Sp(2n,\R)$. Then the geodesic from $U$ in direction $\Delta$ is given by
    \begin{equation}
     \label{eq:RiemGeodesicStiefel}
     \begin{split}
     \Exp^{\SpSt,g}_U(t\Delta) &:= \pi(\Exp^{\Sp,g}_M(t\Delta^{\hor,g}_M))\\
     &=\expm(t(\bar\Omega(\Delta)-\bar\Omega(\Delta)^T))\expm(t\bar\Omega(\Delta)^T)U
     \end{split}
    \end{equation}
    with $\bar\Omega(\Delta)$ as in \eqref{eq:barOmegafromDelta}.
\end{proposition}
\begin{proof}
    This follows directly from the preceding discussion and the definition of the horizontal lift.
\end{proof}

Equation \eqref{eq:RiemGeodesicStiefel} is formulated with $2n\times 2n$-matrices, but
may in practical calculations be reduced to work with tall, skinny $2n\times 8k$ matrices and matrix exponentials of  an $8k \times 8k$ and a $4k \times 4k$ matrix, respectively.
To this end, define $\bar A \in \so(2k,\R)$ by
\begin{equation*}
    \bar A := J_{2k}U^T\Delta(U^TU)^{-1}J_{2k} + (U^TU)^{-1}\Delta^TU - (U^TU)^{-1}\Delta^T J_{2n}^T U(U^TU)^{-1}J_{2k}
\end{equation*}
and define 
\begin{equation*}
    \bar H := (I_{2n}-UU^+)J_{2n}\Delta(U^TU)^{-1}J_{2k}.
\end{equation*}
With $\bar \Delta := U\bar A + \bar H \in T_U\SpSt(2n,2k)$ it holds that $\bar\Omega(\Delta) = YX^T$, where
\begin{align*}
    X := \begin{bmatrix}(I-\frac12 UU^+)\bar\Delta & -U\end{bmatrix} \in \R^{2n\times 4k}
\end{align*}
and
\begin{align*}
    Y := \begin{bmatrix} J_{2n}^T U J_{2k} & (\bar \Delta^+(I_{2n}-\frac12 UU^+))^T  \end{bmatrix} \in \R^{2n\times 4k}.
\end{align*}
This follows from \eqref{eq:barOmegafromDelta} and solving $\bar\Omega(\Delta)^T = (I_{2n}- \frac12 UU^+)\bar\Delta U^+ - U\bar\Delta^+(I_{2n}- \frac12 UU^+)$
for $\bar\Delta$, i.e., $\bar A = U^+\bar\Omega(\Delta)^TU$ and $\bar H = (I_{2n}-UU^+)\bar\Omega(\Delta)^TU$.
Furthermore, define $\hat X := \begin{bmatrix}Y & -X \end{bmatrix} \in \R^{2n\times 8k}$ and  $\hat Y := \begin{bmatrix}X & Y \end{bmatrix} \in \R^{2n\times 8k}$.

\begin{proposition}
    With notation as above, it holds that
    \begin{equation}
    \label{eq:SpStGeodesicSmall}
        \Exp^{\SpSt,g}_U(\Delta) = \hat X \expm(\hat Y^T \hat X)\begin{bmatrix}0_{4k} \\ I_{4k} \end{bmatrix}\expm\left(Y^TX
        \right)\begin{bmatrix}0_{2k}\\ I_{2k}\end{bmatrix}.
    \end{equation} 
\end{proposition}
\begin{proof}
    First, note that since $\bar\Omega(\Delta) = YX^T$, it holds that $\expm(\bar\Omega(\Delta)^T)U = \expm(XY^T)U$. We make use of \eqref{eq:reducedexp}, which implies
    \begin{align*}
        \expm(XY^T)U = U + X(\expm(Y^TX)- I_k)(Y^TX)^{-1}Y^TU.
    \end{align*}
    Since $(Y^TX)^{-1}Y^TU = (Y^TX)^{-1}Y^TX\begin{bsmallmatrix} 0_{2k}\\-I_{2k}\end{bsmallmatrix} = \begin{bsmallmatrix} 0_{2k}\\-I_{2k}\end{bsmallmatrix}$, the simplified expression $\expm(XY^T)U = X\expm(Y^TX)\begin{bsmallmatrix} 0_{2k}\\-I_{2k}\end{bsmallmatrix}$ follows.
    
    Secondly, it holds that $\expm(\bar\Omega(\Delta)-\bar\Omega(\Delta)^T) = \expm(\hat X \hat Y^T)$. Repeating the steps above and noting $(\hat Y^T \hat X)^{-1}\hat Y^T X = \begin{bsmallmatrix}0_{4k} \\ -I_{4k} \end{bsmallmatrix}$ leads to the claimed result.
\end{proof}

\section{The real symplectic Grassmann manifold}
\label{sec:Grassmann}

Similar to the usual Grassmann manifold \cite{BatziesHueperMachadoLeite2015,EdelmanAriasSmith1999} of linear subspaces of a fixed dimension, we define the \emph{real symplectic Grassmann manifold} as the manifold of symplectic subspaces of dimension $2k$ of $(\R^{2n},\omega_0)$. This must not be confused with the Lagrangian Grassmannian, the manifold of Lagrangian subspaces, which is also referred to as the symplectic Grassmann manifold by some authors. The quotient manifold approach which we use is similar to the course of action in \cite{BendokatZimmermannAbsil2020}. As in the case of linear subspaces, we identify a symplectic subspace with the associated symplectic projection onto it.

\begin{proposition}
\label{prop:Grassmann_manifold_structure}
The set
\begin{equation}
\label{eq:SpGr}
    \SpGr(2n,2k) := \{P \in \R^{2n \times 2n} \mid P^2=P,\ \rank(P)=2k,\ P^+=P\}
\end{equation}
consists of the symplectic projections onto the $2k$-dimensional symplectic subspaces of the standard symplectic space $(\R^{2n},\omega_0)$. It has a smooth manifold structure and is called the \emph{real symplectic Grassmann manifold}. It features the quotient representation
\begin{equation}
\SpGr(2n,2k)\cong\Sp(2n,\R)/(\Sp(2k,\R)\times \Sp(2(n-k),\R))
\end{equation}
and has dimension
\begin{equation*}
 \dim\SpGr(2n,2k) = 4(n-k)k.
\end{equation*} 
\end{proposition}
\begin{proof}
We show first that the thus defined space $\SpGr(2n,2k)$ is the orbit of
\begin{equation*}
    E_0:=EE^+
\end{equation*}
under the group action of $\Sp(2n,\R)$ defined by
\begin{equation}
\label{eq:GrassmannGroupAction}
    \phi\colon \Sp(2n,\R) \times \R^{2n \times 2n},\ (M,X) \mapsto MXM^+.
\end{equation}
Because every $U\in\SpSt(2n,2k)$ has a representation $U=ME$, $M\in \Sp(2n,\R)$,
it is sufficient to show that every $P \in \SpGr(2n,2k)$ is equal to $P = UU^+$ for some $U \in \SpSt(2n,2k)$. 
This fact is established as follows:
Since $PJ_{2n}P^T$ is skew-symmetric, it features a `Schur-like decomposition' \cite[eq. (5)]{Xu2003} 
of the form
\begin{equation*}
    PJ_{2n}P^T = Q \begin{bsmallmatrix} 0 & \Sigma^2 & 0\\ -\Sigma^2 &  0 & 0\\ 0 & 0 & 0 \end{bsmallmatrix}Q^T,
\end{equation*}
where $Q \in \O(2n)$ is a real orthogonal matrix.
Moreover $\Sigma = \diag(\sigma_1,\dots,\sigma_k)$, where $\sigma_i > 0$ for all $i=1,\dots,k$, because $\rank(P) = 2k$ \cite[Proposition 3]{Xu2003}. From $P^+ = P = P^2$, it follows that $P =  Q \begin{bsmallmatrix} 0 & \Sigma^2 & 0\\ -\Sigma^2 &  0 & 0\\ 0 & 0 & 0 \end{bsmallmatrix}Q^TJ_{2n}^T$.
For $U := Q I_{2n,2k} \begin{bsmallmatrix} \Sigma & 0\\ 0 & \Sigma\end{bsmallmatrix} \in \R^{2n\times 2k}$, with $I_{2n,2k}$ as in \eqref{eq:I_nk}, it furthermore holds that $P = UU^+$.
The fact that $U^+U=I_{2k}$, i.e.,  $U \in \SpSt(2n,2k)$, follows from $P^2 = P$. 
The other inclusion, i.e., $\phi(M,E_0) \in \SpGr(2n,2k)$ for all $M \in \Sp(2n,\R)$ is immediate.
The stabilizer of the group action $\phi(\cdot,E_0)$ is given by
\begin{equation*}
\begin{split}
    \stab_{E_0}&=\{M \in \Sp(2n,\R) \mid ME_0M^+=E_0\}\\
    &=\left\lbrace
    \begin{bsmallmatrix}
    A_1 & 0 & B_1 & 0\\
    0 & A_2 & 0 & B_2\\
    C_1 & 0 & D_1 & 0\\
    0 & C_2 & 0 & D_2
    \end{bsmallmatrix}
    \in \Sp(2n,\R)\right\rbrace,
\end{split}
\end{equation*}
where $\begin{bsmallmatrix} A_1 & B_1\\ C_1 & D_1\end{bsmallmatrix} \in \Sp(2k,\R)$ and $\begin{bsmallmatrix} A_2 & B_2\\ C_2 & D_2\end{bsmallmatrix} \in \Sp(2(n-k),\R)$. Hence, \begin{equation*}
\stab_{E_0}\cong \Sp(2k,\R) \times \Sp(2(n-k),\R).                                                                                                                                                                                     \end{equation*} 
The manifold structure now follows from \cite[Theorem 21.20]{Lee2012smooth}. The real symplectic Grassmann manifold is by \cite[Theorem 21.18]{Lee2012smooth} diffeomorphic to the homogeneous space
\begin{equation*}
\SpGr(2n,2k)\cong \Sp(2n,\R)/(\Sp(2k,\R)\times \Sp(2(n-k),\R)).
\end{equation*}
The dimension of $\SpGr(2n,2k)$ is obtained via the standard formula
 \begin{align*}
  \dim\SpGr(2n,2k) &= \dim\Sp(2n,\R) - \dim\Sp(2k, \R) \cdot \dim\Sp(2(n-k),\R)\\
  & =  4(n-k)k.
 \end{align*}
\end{proof}

Note that the real symplectic Grassmann manifold $\SpGr(2n,2k)$ has the same dimension as the Grassmann manifold $\Gr(2n,2k)$. The manifolds are not the same however, since not every $2k$ dimensional subspace of $\R^{2n}$ is also a symplectic subspace of the standard symplectic space $(\R^{2n},\omega_0)$.

Similarly to the Grassmann case \cite{BatziesHueperMachadoLeite2015}, for every $P \in \SpGr(2n,2k)$, we can define a set
\begin{equation*}
\begin{split}
 \sp_P(2n) &:= \left\lbrace \tilde\Omega \in \sp(2n,\R) \ \middle|\ \tilde \Omega = \tilde\Omega P + P\tilde\Omega \right\rbrace\\
 &= \left\lbrace M\Omega M^+ \in \sp(2n,\R) \ \middle|\ P = M E_0 M^+,\ \Omega \in \sp_{E_0}(2n) \right\rbrace .
\end{split}
\end{equation*}

The tangent space of $\SpGr(2n,2k)$ at $P$ is characterized by the following proposition.
\begin{proposition}
\label{prop:TangentSpaceGrassmann}
    Let $P\in \SpGr(2n,2k)$. The tangent space at $P$ is given by
    \begin{equation}
    \begin{split}
        T_P\SpGr(2n,2k)&=\{ [\Omega,P] \mid \Omega \in \sp(2n,\R)\}\\
        &=\{ [\tilde\Omega,P] \mid \tilde\Omega \in \sp_P(2n) \}.
    \end{split}
    \end{equation}
\end{proposition}
    \begin{proof}
     The first equality follows by a straightforward calculation, as every tangent vector is the derivative of a curve defined via \eqref{eq:GrassmannGroupAction}. The second equality follows from $\{ [\tilde\Omega,P] \mid \tilde\Omega \in \sp_P(2n) \} \subset \{ [\Omega,P] \mid \Omega \in \sp(2n,\R)\}$ and the fact that for every $\Omega \in \sp(2n,\R)$, it holds that $\tilde\Omega := \Omega P + P \Omega - 2P \Omega P \in \sp_P(2n)$ and $[\Omega, P] = [\tilde\Omega,P]$.
    \end{proof}

\subsection{Pseudo-Riemannian metric on \texorpdfstring{$\SpGr(2n,2k)$}{SpGr(2n,2k)}}

We can connect the real symplectic Stiefel manifold, i.e. the manifold of symplectic bases, with the real symplectic Grassmann manifold in the following way.
\begin{proposition}
 The map
 \begin{equation}
\label{eq:projStGr}
\rho\colon \SpSt(2n,2k) \to \SpGr(2n,2k), \hspace{0.2cm}
U \mapsto \rho(U):=UU^+
\end{equation}
is a surjective submersion. Every tangent space $T_U\SpSt(2n,2k)$ splits into a vertical and horizontal part with respect to $\rho$ and the pseudo-Riemannian metric $h^\SpSt_U$, namely $T_U\SpSt(2n,2k)= \Ver^\rho_U\SpSt(2n,2k) \oplus \Hor^{\rho,h}_U\SpSt(2n,2k)$, where
\begin{equation*}
\label{eq:VerStiefel}
    \Ver^\rho_U\SpSt(2n,2k) := \ker \D \rho_U = \{ UA \mid A \in \sp(2k,\R) \}
\end{equation*}
and
\begin{equation*}
    \Hor^{\rho,h}_U\SpSt(2n,2k) := (\ker \D \rho_U)^{\perp,h^\SpSt_U} = \{ U^sB \mid B \in \R^{2(n-k)\times 2k} \}.
\end{equation*} 
\end{proposition}
\begin{proof}
This is a standard construction. We only show that $\rho$ is a surjective submersion. As $\SpGr(2n,2k)$ is the orbit of $E_0$ under the group action $\phi$ of \eqref{eq:GrassmannGroupAction}, for every $P \in \SpGr(2n,2k)$ there is $M \in \Sp(2n,\R)$  such that $P = MEE^+M^+$. Defining $U = ME \in \SpSt(2n,2k)$ shows that $P = \rho(U)$ and therefore that the map $\rho$ is surjective. To show that $\rho$ is a submersion, we show that the differential $\D \rho_U$ is surjective for every $U \in \SpSt(2n,2k)$: Let $P = UU^+ \in \SpGr(2n,2k)$ and $[\Omega,P] \in T_P\SpGr(2n,2k)$. Then $\Delta := \Omega U \in T_U\SpSt(2n,2k)$ and $\D \rho_U(\Delta) = \Delta U^+ + U\Delta^+ =\Omega UU^+ - UU^+\Omega = [\Omega,P]$.
\end{proof}
Let $M \in \Sp(2n,\R)$ such that $P=ME_0 M^+$ and define $U=ME \in \SpSt(2n,2k)$ and $U^s=ME^s$. Then $P=UU^+$, and it follows from the preceding proposition that $T_P\SpGr(2n,2k)$ can be identified with $\Hor^{\rho,h}_U\SpSt(2n,2k)$ via the horizontal lift. For $\Gamma \in T_P\SpGr(2n,2k)$, this horizontal lift is explicitly given by
\begin{equation*}
 \Gamma^\hor_U =  \Gamma U \in \Hor^{\rho,h}_U\SpSt(2n,2k),
\end{equation*}
as can be seen by the fact that there is $\Omega \in \sp(2n,\R)$ such that $\Gamma = [\Omega, UU^+]$ and
\begin{equation*}
\begin{split}
    \D \rho_U(\Gamma U) &= \Gamma UU^+ + UU^+\Gamma^+ = \Omega UU^+ - UU^+\Omega UU^+ + UU^+\Omega^+ - UU^+\Omega^+ UU^+\\
    &= \Omega UU^+ - UU^+\Omega= [\Omega, UU^+] = \Gamma.
\end{split}
\end{equation*}

By making use of the horizontal lift to $\SpSt(2n,2k)$, we can define a pseudo-Riemannian metric on the real symplectic Grassmann manifold $\SpGr(2n,2k)$.
\begin{proposition}
 Let $\Gamma_1,\Gamma_2 \in T_P\SpGr(2n,2k)$ and $U \in \rho^{-1}(P)$. There is $B_i \in \R^{2(n-k)\times 2k}$ such that ${(\Gamma_i)}^\hor_U=U^s B_i$, i=1,2. The mapping $g^\SpGr_P\colon T_P\SpGr(2n,2k) \times T_P\SpGr(2n,2k) \to \R$,
    \begin{equation}
    \label{eq:pRmetricGr}
        g^\SpGr_P(\Gamma_1,\Gamma_2) := h^\SpSt_U\left((\Gamma_1)^\hor_U,(\Gamma_2)^\hor_U\right)
        = \tr\left(U^+\Gamma_1^+\Gamma_2U\right) = \tr(B_1^+B_2)
    \end{equation} 
 defines point-wise a pseudo-Riemannian metric on $\SpGr(2n,2k)$. 
\end{proposition}
 \begin{proof}
  Similar to Proposition~\ref{prop:pRmetricSt}.
 \end{proof}
In contrast to $\SpSt(2n,2k)$, which is a naturally reductive space, $\SpGr(2n,2k)$ is even \emph{symmetric} with respect to the pseudo-Riemannian metric \eqref{eq:pRmetricGr}. 
To see this, let $X= \diag(-I_k,I_{n-k},-I_k,I_{n-k})$ block-diagonal
and observe that the involutive automorphism $\sigma\colon \Sp(2n,\R) \to \Sp(2n,R)$, $\sigma(M)=XMX$ fulfills \cite[Theorem 11.29]{ONeill1983}.
By \cite[Lemma 11.24]{ONeill1983}, it therefore holds that $\rho \circ \pi$ is a pseudo-Riemannian submersion with respect to \eqref{eq:pRmetricGr}, since any symmetric space is naturally reductive \cite[p. 317]{ONeill1983}.

The connection between $\SpSt(2n,2k)$ and $\SpGr(2n,2k)$ allows us to state the following decomposition of real symplectic Stiefel matrices.
\begin{corollary}
    Every $U \in \SpSt(2n,2k)$ is of the form
    \begin{equation*}
        U = Y\diag(\Sigma,\Sigma)N,
    \end{equation*}
    where $N \in \Sp(2k,\R)$, $\Sigma = \diag(\sigma_1,\dots,\sigma_k)$ with $\sigma_i > 0$ for all $i = 1,\dots, k$ and $Y \in \St(2n,2k)$ fulfills $Y^+Y = \diag(\Sigma,\Sigma)^{-2}$.
\end{corollary}
\begin{proof}
    As in the proof of Proposition~\ref{prop:Grassmann_manifold_structure}, it follows from the Schur-like decomposition \cite[Equation (5)]{Xu2003} that the matrix $P := UU^+ \in \SpGr(2n,2k)$ is of the form $P = \tilde U \tilde U^+$, where $\tilde U = Q I_{2n,2k}\diag(\Sigma,\Sigma) \in \SpSt(2n,2k)$ and $Q \in \O(2n)$ is orthogonal. Define $Y:= QI_{2n,2k}$. It holds that $Y^TY = I_{2k}$, so $Y \in \St(2n,2k)$, and from $\tilde U^+\tilde U = I_{2k}$ it follows that $Y^+Y = \diag(\Sigma,\Sigma)^{-2}$. The claim now follows from the fact that $U = \tilde U N$ for some $N \in \Sp(2k,\R)$. 
\end{proof}


As the metric $g^\SpGr$ of \eqref{eq:pRmetricGr} is defined via a horizontal lift, we obtain the associated geodesics by projection.
\begin{proposition}
 Let $P \in \SpGr(2n,2k)$ and $\Gamma \in T_P\SpGr(2n,2k)$. Furthermore, let $U \in \rho^{-1}(P) \subset \SpSt(2n,2k)$. The geodesic starting at $P$ in direction $\Gamma$ with respect to the metric \eqref{eq:pRmetricGr} is
 \begin{equation}
 \label{eq:expSpGr}
    \Exp^\SpGr_P(t\Gamma):= \rho(\Exp^{\SpSt,h}_U(t\Gamma^\hor_U))=\expm(t[\Gamma,P])P\expm(-t[\Gamma,P]).
 \end{equation} 
\end{proposition}
 \begin{proof}
  By \cite[Proposition 11.31]{ONeill1983}, the pseudo-Riemannian geodesics on the real symplectic Grassmannian $\SpGr(2n,2k)$ are the projections of the one-parameter subgroups in $\Sp(2n,\R)$ under the pseudo-Riemannian submersion $\rho\circ \pi$. Since
  \begin{equation*}
   \rho( \Exp^{\SpSt,h}_U(\Gamma^\hor_U)) = (\rho \circ \pi)\left(\Exp^{\Sp,h}_M\left((\Gamma^\hor_U)^\hor_M\right)\right),
  \end{equation*}
where $M \in \pi(U)^{-1}$, the claim follows.
 \end{proof}
By making use of Proposition~\ref{prop:StiefelReducedExp}, we can reduce the computational complexity of \eqref{eq:expSpGr}. To this end note that $H:=\Gamma^\hor_U \in \Hor^{\rho,h}_U\SpSt(2n,2k)$. Therefore, if $H^+H$ is invertible,
\begin{equation*}
    \Exp^{\SpSt,h}_U(t\Gamma^\hor_U) = \begin{bmatrix}- H & U \end{bmatrix}\expm\left(t\begin{bmatrix} 0 & -I_{2k}\\ H^+H & 0 \end{bmatrix}\right)\begin{bmatrix} 0 \\ I_{2k}\end{bmatrix},
\end{equation*} 
which implies
\begin{equation}\label{eq:GrassmannreducedExp}
\begin{split}
    \Exp^\SpGr_P(t\Gamma) &= \Exp^{\SpSt,h}_U(t\Gamma^\hor_U)(\Exp^{\SpSt,h}_U(t\Gamma^\hor_U))^+\\
    &=\begin{bsmallmatrix}- H & U \end{bsmallmatrix}\expm\left(t\begin{bsmallmatrix} 0 & -I_{2k}\\ H^+H & 0 \end{bsmallmatrix}\right)\begin{bsmallmatrix} 0 & 0\\0 & I_{2k}\end{bsmallmatrix} \expm\left(t\begin{bsmallmatrix} 0 & H^+H\\ -I_{2k} & 0 \end{bsmallmatrix}\right) \begin{bsmallmatrix} -H^+\\ U^+\end{bsmallmatrix}.
\end{split}
\end{equation}
Lifting to another representative $\tilde{U} = UN$ of $P$, where $N \in \Sp(2k,\R)$, implies $\tilde H = \Gamma^\hor_{UN}=HN$, with which one can check that \eqref{eq:GrassmannreducedExp} does not depend on the chosen representative.

Finding the (local) inverse of \eqref{eq:expSpGr}, i.e. given two points $P,F \in \SpGr(2n,2k)$, find the tangent vector $\Gamma \in T_P\SpGr(2n,2k)$ such that $\Exp^\SpGr_P(\Gamma)=F$, is called the geodesic endpoint problem, or also pseudo-Riemannian logarithm. The structure of the real symplectic Grassmann manifold allows us to find it similarly to the case of the standard Grassmann manifold \cite[Theorem 3.3]{BatziesHueperMachadoLeite2015}.
\begin{proposition}\label{prop:GrassmannLog}
    Let $P,F \in \SpGr(2n,2k)$. If 
    \begin{equation}
    \label{eq:SympGrlog}
      \tilde \Omega = \frac12\logm\left((I_{2n}-2F)(I_{2n}-2P)\right)
    \end{equation}
    is well defined and $\tilde \Omega \in \sp_P(2n)$, it holds for $\Gamma := [\tilde \Omega,P] \in T_P\SpGr(2n,2k)$ that $\Exp^\SpGr_P(\Gamma)=F$. 
\end{proposition}
    \begin{proof}
        We have to show that $F = \expm([\Gamma,P])P\expm(-[\Gamma,P])$. Since by assumption $\tilde \Omega \in \sp_P(2n)$ and therefore $[\Gamma,P] = \tilde \Omega$, this is equivalent to showing $F = \expm(\tilde\Omega) P \expm(-\tilde\Omega)$. The fact that $\Gamma \in T_P\SpGr(2n,2k)$ holds by Proposition~\ref{prop:TangentSpaceGrassmann}.\\
        For $\tilde\Omega$ defined in \eqref{eq:SympGrlog}, it holds that $(I_{2n}-2P)\tilde \Omega (I_{2n}-2P) = \frac12\logm((I_{2n}-2P)(I_{2n}-2F)) = -\tilde \Omega$, since $(I_{2n}-2P)^{-1}= (I_{2n}-2P)$. Therefore $(I_{2n}-2P)\expm(-\tilde\Omega) = (I_{2n}-2P)\expm(-\tilde\Omega)(I_{2n}-2P)^2 = \expm(\tilde\Omega)(I_{2n}-2P)$.
        This leads to $\expm(\tilde\Omega)P\expm(-\tilde\Omega) = \frac12 I_{2n} + \expm(2\tilde\Omega)(-\frac12 I_{2n} + P) = F$, which shows the claim.
    \end{proof}
    
\subsection{Riemannian metric on \texorpdfstring{$\SpGr(2n,2k)$}{SpGr(2n,2k)}}
\label{subsec:SymplGrassRiemannianMetric}
As the real symplectic Grassmann manifold $\SpGr(2n,2k)$ is a quotient of $\SpSt(2n,2k)$ (and of $\Sp(2n,\R)$), we can obtain a Riemannian metric from a right-invariant Riemannian metric on $\SpSt(2n,2k)$.

Again, we split the tangent space $T_U\SpSt(2n,2k)$ at $U \in \SpSt(2n,2k)$ into a vertical part with respect to $\rho$ and a horizontal part with respect to $\rho$ and $g^\SpSt_U$ from \eqref{eq:RiemMetricStiefel}. The former yields \eqref{eq:VerStiefel}, as the vertical space is independent of the metric. The latter gives
\begin{equation*}
\begin{split}
 \Hor^{\rho,g}_U\SpSt(2n,2k) &= (\Ver^\rho_U\SpSt(2n,2k))^{\perp,g}\\
 &= \left\lbrace (UH^+ - HU^+)^TU \ \middle|\ U^+H = 0 \right\rbrace.
\end{split}
\end{equation*}
This follows from $g^\SpSt_U\left(UA,(UH^+ - HU^+)^TU\right) = 0$ for all $UA \in \Ver^\rho_U\SpSt(2n,2k)$ and all $H \in \R^{2n\times 2k}$ with $U^+H = 0$, and by counting degrees of freedom. For any $\Delta = (UH^+ - HU^+)^TU \in \Hor^{\rho,g}_U\SpSt(2n,2k)$, the corresponding $H$ (which is not to be confused with $(I_{2n}-UU^+)\Delta$ here) can by calculated via
\begin{equation}
\label{eq:HfromhorizontalDelta}
    H = (I_{2n}-UU^+)J_{2n}^T\Delta(U^TU)^{-1}J_{2k}.
\end{equation} 

We can identify the horizontal space $\Hor^{\rho,g}_U\SpSt(2n,2k)$ with the tangent space $T_{\rho(U)}\SpGr(2n,2k)$ and define a Riemannian metric on $\SpGr(2n,2k)$ via the restriction of the Riemannian metric $g^\SpSt$
of \eqref{eq:RiemMetricStiefel} to the horizontal spaces.

For $\Gamma_i \in T_{UU^+}\SpGr(2n,2k)$, let $(\Gamma_i)^{\hor,g}_U \in \Hor^{\rho,g}_U\SpSt(2n,2k)$ be the horizontal lift to $\Hor^{\rho,g}_U\SpSt(2n,2k)$, i.e. $\D \rho_U((\Gamma_i)^{\hor,g}_U) = \Gamma_i$. The mapping
\begin{equation*}
 \begin{split}
  g^\SpGr_{UU^+}&\colon T_{UU^+}\SpGr(2n,2k) \times T_{UU^+}\SpGr(2n,2k) \to \R,\\
  g^\SpGr_{UU^+}(\Gamma_1,\Gamma_2)&:=g^\SpSt_U((\Gamma_1)^{\hor,g}_U,(\Gamma_2)^{\hor,g}_U)
 \end{split}
\end{equation*} 
defines pointwise a Riemannian metric. We are not aware of an explicit mapping to calculate the horizontal lift with respect to $g^\SpSt$ for a given $\Gamma \in T_{UU^+}\SpGr(2n,2k)$. Nevertheless, one can directly work with symplectic Stiefel representatives and horizontal tangent vectors, i.e., with $2n \times 2k$-matrices.

\begin{lemma}
 For two horizontal tangent vectors 
 \begin{equation*}
    \Delta_i := (\Gamma_i)^{\hor,g}_U = (UH_i^+ - H_iU^+)^TU \in \Hor^{\rho,g}_U\SpSt(2n,2k),
 \end{equation*}
it holds that
\begin{equation*}
\begin{split}
 g^\SpSt_U((\Gamma_1)^{\hor,g}_U,(\Gamma_2)^{\hor,g}_U) &= \tr\left((U^TU)^{-1}\Delta_1^T(I_{2n}-UU^+)\Delta_2\right)\\
 &= \tr\left(U^TU(H_2^TH_1)^+ - (U^TH_1)^+H_2^TU\right).
\end{split}
 \end{equation*} 
\end{lemma}
\begin{proof}
 This follows by a direct calculation from the properties of the trace.
\end{proof}

Let $f$ be a function on the real symplectic Grassmannian, given on symplectic Stiefel representatives by $f\colon \SpSt(2n,2k) \to \R$, with $f(U)=f(UN)$ for all $N \in \Sp(2k,\R)$. 
We assume that $f$ can (locally) be extended to a smooth function on $2n\times 2k$-matrices, for convenience again denoted by $f$.
The Riemannian gradient of $f$ with respect to $g^\SpGr$ is given by
\begin{equation*}
    \grad^g_f(U) = (UH^+ - HU^+)^TU = J_{2n}^THJ_{2k}U^TU - J_{2n}^TUJ_{2k}H^TU,
\end{equation*} 
with
\begin{equation*}
    H = (I_{2n}-UU^+)J_{2n}^T\nabla f_U J_{2k},
\end{equation*}
where $\nabla f_U$ denotes the Euclidean gradient of a smooth extension of $f$ around $U$ in $\R^{2n\times 2k}$. This follows from \cite[Equation (3.39)]{AbsilMahonySepulchre2008} and $g^\SpSt_U(\grad^g_f(U),\Delta) = \D f_U(\Delta) = \tr((\nabla f_U)^T\Delta)$ for all $\Delta \in \Hor^{\rho,g}_U\SpSt(2n,2k)$, as well as the fact that $\grad^g_f(U) \in \Hor^{\rho,g}_U\SpSt(2n,2k)$.

\begin{proposition}
 Let $U \in \SpSt(2n,2k)$ and $\Delta \in \Hor^{\rho,g}_U\SpSt(2n,2k)$. The lifted symplectic Grassmann geodesic from $U$ in direction $\Delta$ is given by
 \begin{equation}
 \label{eq:RiemExpSpGrBig}
   \Exp^\SpGr_U(t\Delta) = \expm(t(\bar\Omega-\bar\Omega^T))\expm(t\bar\Omega^T)U,
 \end{equation}
where $\bar\Omega$ is given by \eqref{eq:barOmegafromDelta}.
\end{proposition}
\begin{proof}
    We need to show that the tangent vector $\frac{\D}{\D t} \Exp^\SpGr_U(t\Delta)$ is horizontal for every $t$. Then, the claim follows from \cite[Cor. 7.46]{ONeill1983}. Since $\Delta \in \Hor^{\rho,g}_U\SpSt(2n,2k)$ is equivalent to $J_{2n}UJ_{2k}^TU^T\bar\Omega(\Delta)J_{2n}UJ_{2k}^TU^T = 0$, the proof follows in the same fashion as the one of Lemma~\ref{lemma:HorStiefelGeodesic}.
\end{proof}

Since $\bar\Omega^T = UH^+-HU^+$, with $H$ from \eqref{eq:HfromhorizontalDelta}, we can reduce  \eqref{eq:RiemExpSpGrBig} with \eqref{eq:reducedexp} to the matrix exponentials of a $8k \times 8k$ and $4k \times 4k$ matrix, respectively.

\begin{proposition}
    Let $U \in \SpSt(2n,2k)$ and $\Delta \in \Hor^{\rho,g}_U\SpSt(2n,2k)$. Define $H$ as in \eqref{eq:HfromhorizontalDelta},
    \begin{equation*}
        X := \begin{bmatrix} J_{2n}^T H J_{2k} & -J_{2n}^T U J_{2k} & -U & H \end{bmatrix} \in \R^{2n \times 8k}
    \end{equation*}
    and
    \begin{equation*}
        Y := \begin{bmatrix} U & H & J_{2n}^T H J_{2k} & J_{2n}^T U J_{2k} \end{bmatrix} \in \R^{2n \times 8k}.
    \end{equation*}
    Then
    \begin{equation*}
      \Exp^\SpGr_U(t\Delta) = X \expm(tY^TX) \begin{bmatrix} 0_{4k} \\ -I_{4k} \end{bmatrix} \expm\left(t \begin{bmatrix} 0 & -H^+H\\ I_{2k} & 0 \end{bmatrix}\right) \begin{bmatrix} I_{2k}\\ 0 \end{bmatrix}.
    \end{equation*} 
\end{proposition}
\begin{proof}
    It holds that $\bar\Omega - \bar\Omega^T = XY^T$ and $\bar\Omega^T = UH^+-HU^+$. By \eqref{eq:reducedexp},
    \begin{align*}
        \expm(\bar\Omega-\bar\Omega^T)\begin{bmatrix} U & -H\end{bmatrix} &= \expm(XY^T)\begin{bmatrix} U & -H\end{bmatrix}\\
        &= (I_{2n} + X(\expm(Y^TX) - I_{8k})(Y^TX)^{-1}Y^T)\begin{bmatrix} U & -H\end{bmatrix}.
    \end{align*}
    Since
    \begin{align*}
        (Y^TX)^{-1}Y^T\begin{bmatrix} U & -H\end{bmatrix} = (Y^TX)^{-1}Y^TX\begin{bmatrix} 0_{4k} \\ -I_{4k} \end{bmatrix},
    \end{align*}
    it follows that $\expm(\bar\Omega -\bar\Omega^T)\begin{bmatrix} U & -H\end{bmatrix} = X\expm(Y^TX)\begin{bsmallmatrix} 0_{4k} \\ -I_{4k} \end{bsmallmatrix}$. Together with
    \begin{align*}
        \expm(\bar\Omega^T)U = \begin{bmatrix} U & -H\end{bmatrix}\expm\left(t \begin{bmatrix} 0 & -H^+H\\ I_{2k} & 0 \end{bmatrix}\right) \begin{bmatrix} I_{2k}\\ 0 \end{bmatrix}
    \end{align*}
    as in Proposition~\ref{prop:StiefelReducedExp}, this shows the claim.
\end{proof}

\section{Retractions and computational issues}
\label{sec:Retractions}

Calculating the matrix exponential of an $n \times n$ matrix is computationally expensive if $n$ is large. Furthermore, numerical experiments show that even though the matrix exponential of a Hamiltonian matrix is theoretically guaranteed to yield a symplectic matrix as an output, this is not necessarily the case in practice, where one needs to rely on numerical tools to compute the standard matrix exponential. 
While there are specialized algorithms for the matrix exponential of a Hamiltonian  matrix \cite{Kuo2019}, there is another alternative: The Cayley map.
In this section, we propose the use of the Cayley map for approximating the pseudo-Riemannian geodesics in order to define retractions on the symplectic Stiefel and Grassmann manifold. Furthermore, these retractions turn out to be invertible in closed form on both manifolds, which can for example be used for interpolation and optimization purposes and for defining local coordinates.
In the experiments of Section \ref{sec:NumericalTests}, the Cayley-based retraction turns out to be computationally cheaper and to retain the manifold structure 
to a much higher numerical accuracy.

A \emph{retraction} \cite{AbsilMahonySepulchre2008} on a smooth manifold $M$ with tangent bundle $TM$ is a smooth mapping $R \colon TM \to M$ such that for any $x\in M$,
\begin{enumerate}
 \item $R_x(0)=x$,
 \item $\D (R_x)_0=\id$,
\end{enumerate}
where $R_x$ is the restriction of $R$ to $T_xM$.

The Cayley transformation
\begin{equation*}
    \cay(X) := (I_n + X)(I_n-X)^{-1},\quad X \in \R^{n \times n}
\end{equation*} 
is widely used as a standard approximation of the matrix exponential $\expm(2X)$.
Of special interest in the present context is the property that $\cay$ maps from $\sp(2n,\R)$ to $\Sp(2n,\R)$ \cite{Arnold2001}. This was also exploited in \cite{fiori2011}. The inverse of the Cayley transform is given by \cite{Arnold2001}
\begin{equation*}
    \cay^{-1}(M) = (M - I_n)(I_n + M)^{-1}.
\end{equation*}

\subsection{Cayley retraction on the real symplectic Stiefel manifold}
Replacing the matrix exponential in the pseudo-Riemannian exponential \eqref{eq:StiefelBigExp}
on the symplectic Stiefel manifold with the Cayley transform leads to the Cayley retraction defined in \cite[Definition 5.2]{GaoSonAbsilStykel2020symplecticoptimization}. 
Yet note that the Cayley retraction in the aforementioned reference was found unaware of the pseudo-Riemannian geodesics by transferring the Cayley retraction on the classical Stiefel manifold $\St(n,k)$ to the symplectic case.
\begin{proposition}
    Let $U \in \SpSt(2n,2k)$ and $\Delta \in T_U\SpSt(2n,2k)$. For $\Delta=\tilde\Omega(U,\Delta) U$, with $\tilde\Omega(U,\Delta)$ as in \eqref{eq:OmegatildefromDelta}, the map
    \begin{equation}\label{eq:RSpStBig}
        \Ret^\SpSt_U(\Delta) := \cay\left(\frac12 \tilde\Omega(U,\Delta)\right)U
    \end{equation}
    is a retraction. The derivative of the curve $\gamma(t):= \Ret^\SpSt_U(t\Delta)=\cay(\frac{t}{2}\tilde\Omega(U,\Delta))U$ is given by
    \begin{equation*}
        \dot{\gamma}(t)= \frac12\left(\left(I_{2n}+\frac{t}{2}\tilde\Omega(U,\Delta)\right)^{-1} + \left(I_{2n}-\frac{t}{2}\tilde\Omega(U,\Delta)\right)^{-1}\right)\tilde\Omega(U,\Delta) \gamma(t)
    \end{equation*} 
\end{proposition}
\begin{proof}
     The fact that $\Ret^\SpSt$ is a retraction is shown in \cite[Prop. 5.3]{GaoSonAbsilStykel2020symplecticoptimization}. The formula for $\dot\gamma(t)$ follows from a straightforward calculation, making use of the fact that $\tilde\Omega(U,\Delta)$ commutes with $\cay(\frac{t}{2}\tilde\Omega(U,\Delta))$.
\end{proof}

In \cite[Proposition 5.5]{GaoSonAbsilStykel2020symplecticoptimization}, it was proposed to use the Sherman-Morrison-Woodbury formula
\begin{equation*}
    (A + XY^T)^{-1} = A^{-1} - A^{-1}X(I+Y^TA^{-1}X)^{-1}Y^TA^{-1},
\end{equation*}
where $A \in \R^{n\times n}, X,Y \in \R^{n \times k}$, to reduce the matrix inverse in \eqref{eq:RSpStBig} from $2n \times 2n$ to $4k \times 4k$. We show that we can even reduce it to a matrix inversion of dimensions  $2k \times 2k$.
\begin{proposition}
    Let $U \in \SpSt(2n,2k)$ and $\Delta \in T_U\SpSt(2n,2k)$. Define $A := U^+\Delta$ and $H := \Delta - UA$. Then
    \begin{equation}\label{eq:RSpStSmall}
        \Ret^\SpSt_U(t\Delta) = -U + (tH+2U)\left(\frac{t^2}{4}H^+H -\frac{t}{2}A + I_{2k}\right)^{-1}.
    \end{equation} 
\end{proposition}
\begin{proof}
    For $t=0$, the equality holds. In the following, assume $t\neq 0$. Similarly to the proof of Proposition~\ref{prop:StiefelReducedExp}, define $X =  \begin{bmatrix} \frac12 UA + H & -U\end{bmatrix} \in \R^{2n \times 4k}$ and $Y^T = \begin{bsmallmatrix} U^+ \\ \Delta^+(I_{2n}-\frac12 UU^+) \end{bsmallmatrix} \in \R^{4k \times 2n}$. Then  again $\tilde\Omega (U,\Delta)= XY^T$ as in \eqref{eq:OmegatildefromDelta}, and
    \begin{align*}
        Y^TX = \begin{bmatrix} \frac12 A & -I_{2k}\\ H^+H -\frac14A^2 & \frac12 A \end{bmatrix}.
    \end{align*}
    By definition
    \begin{align*}
        \Ret^\SpSt_U(t\Delta) &= \cay\left(\frac{t}{2}\tilde\Omega(U,\Delta)\right)U = \left(I_{2n}+\frac{t}{2}XY^T\right)\left(I_{2n}-\frac{t}{2}XY^T\right)^{-1}U\\
        &=\left(I_{2n}+\frac{t}{2}XY^T\right)\left(I_{2n} + \frac{t}{2}X(I_{4k}-\frac{t}{2}Y^TX)^{-1}Y^T\right)U\\
        &= U + tX(I_{4k}-\frac{t}{2}Y^TX)^{-1}Y^TU.
    \end{align*}
    It holds that
    \begin{align*}
        I_{4k}-\frac{t}{2}Y^TX = \begin{bmatrix} I_{2k} - \frac{t}{4}A & \frac{t}{2}I_{2k}\\-\frac{t}{2}(H^+H-\frac14A^2) & I_{2k}-\frac{t}{4}A \end{bmatrix}.
    \end{align*}
    Block-matrix inversion via the Schur complement yields
    \begin{align*}
        (I_{4k}-\frac{t}{2}Y^TX)^{-1} =
        \begin{scriptstyle}
        \begin{bmatrix}
            -\frac{1}{2}\Theta^{-1}(\frac{t}{2}A-2I_{2k}) & -\frac{t}{2}\Theta^{-1}\\
            -\frac{1}{2t}(\frac{t}{2}A-2I_{2k})\Theta^{-1}(\frac{t}{2}A - 2I_{2k})+\frac{2}{t}I_{2k} & -\frac{1}{2}(\frac{t}{2}A-2I_{2k})\Theta^{-1}
        \end{bmatrix}
        \end{scriptstyle}
    \end{align*}
    with $\Theta = \frac{t^2}{4}H^+H - \frac{t}{2}A + I_{2k}\in\R^{2k\times 2k}$.  
    Writing $Y^TU = \begin{bmatrix}I_{2k}\\ -\frac12 A\end{bmatrix}$ it follows that
    \begin{align*}
        (I_{4k}-\frac{t}{2}Y^TX)^{-1}Y^TU =
        \begin{bmatrix}
            \Theta^{-1}\\
            \frac{1}{t}(\frac{t}{2}A-2I_{2k})\Theta^{-1} + \frac{2}{t}I_{2k}
        \end{bmatrix}.
    \end{align*}
    Putting everything together, we obtain
    \begin{align*}
        \Ret^\SpSt_U(t\Delta) &= U + t(\frac{1}{2}UA+H - \frac{1}{t}U(\frac{t}{2}A - 2I_{2k} + \frac{t^2}{2}H^+H -tA + 2I_{2k}))\Theta^{-1}\\
        &= -U + (tH + 2U)(\frac{t^2}{4}H^+H - \frac{t}{2}A + I_{2k})^{-1},
    \end{align*}
    which shows the claim.
\end{proof}

Unlike the pseudo-Riemannian exponential \eqref{eq:StiefelReducedExp1} or the Riemannian exponential \eqref{eq:SpStGeodesicSmall}, we can invert the Cayley retraction \eqref{eq:RSpStSmall} in closed form. Apart from interpolation, this facilitates the calculation of local coordinates on $\SpSt(2n,2k)$.

\begin{proposition}\label{prop:Inv_Stiefel_Retraction}
    Let $U, V \in \SpSt(2n,2k)$. If $(I_{2k} + U^+V)^{-1}$ and $(I_{2k} + V^+U)^{-1}$ exist, it holds for
    \begin{equation*}
        A = 2((I_{2k} + V^+U)^{-1} - (I_{2k} + U^+V)^{-1}) \in \sp(2k,\R)
    \end{equation*}
    and
    \begin{equation*}
        H = 2((V + U)(I_{2k} + U^+V)^{-1} - U) \in \Hor^{\rho,h}_U\SpSt(2n,2k),
    \end{equation*}
    that
    \begin{equation}
    \label{eq:InverseStiefelRetraction}
        \L^\SpSt_U(V) := UA + H  \in T_U\SpSt(2n,2k)
    \end{equation}
    fulfills $\Ret^\SpSt_U(\L^\SpSt_U(V))= V$.
\end{proposition}
\begin{proof}
    Since $A^+=-A$, it holds that $A \in \sp(2k,\R)$ and $U^+H = 2(U^+V + I_{2k})(I_{2k} + U^+V)^{-1} - I_{2k}) = 0$ implies $H \in \Hor^{\rho,h}_U\SpSt(2n,2k)$. Therefore, $\L^\SpSt_U(V)$ is a valid tangent vector. Since
    \begin{align*}
        \frac14 H^+H = (U^+V+I_{2k})^{-1} + (V^+U+I_{2k})^{-1} - I_{2k},
    \end{align*}
    it holds that
    \begin{align*}
     \frac14 H^+H - \frac12 A + I_{2k} = 2(U^+V+I_{2k})^{-1}.
    \end{align*}
    Therefore
    \begin{align*}
        \Ret^\SpSt_U(\L^\SpSt_U(V)) &= -U + (H+2U)(\frac14 H^+H - \frac12 A + I_{2k})^{-1}\\
        &= -U +\frac12 H(U^+V + I_{2k}) + U(U^+V + I_{2k}) = V,
    \end{align*}
    which shows the claim.
\end{proof}

\subsection{Cayley retraction on the real symplectic Grassmann manifold}
With the quotient manifold approach to the symplectic Stiefel manifold and the definition of the symplectic Grassmann manifold, we can show an additional property of $\Ret^\SpSt$: It maps horizontal tangent vectors (with respect to the pseudo-Riemannian metric $h^\SpSt$ from \eqref{eq:pRmetricSt}) to curves with horizontal tangent vectors everywhere. We can therefore use it to calculate approximations of the pseudo-Riemannian symplectic Grassmann geodesics lifted to the symplectic Stiefel manifold.
\begin{proposition}\label{prop:horizontalretraction}
    Let $U \in \SpSt(2n,2k)$ and $\Delta \in \Hor^{\rho,h}_U\SpSt(2n,2k)$. Furthermore, let $\tilde\Omega(U,\Delta)$ is as in \eqref{eq:OmegatildefromDelta}. For
    \begin{equation*}
        \gamma(t):= \Ret^\SpSt_U(t\Delta)=\cay\left(\frac{t}{2}\tilde\Omega(U,\Delta)\right)U,
    \end{equation*}
    it holds that $\dot\gamma(t) \in \Hor^{\rho,h}_{\gamma(t)}\SpSt(2n,2k)$ for all $t$.
    \begin{proof}
        We suppress the dependence of $\tilde\Omega$ on $U$ and $\Delta$ for better legibility. We have to show that $\gamma(t)^+\dot\gamma(t)=0$ for all $t$. It holds that $\cay(\frac{t}{2}\tilde\Omega)$ commutes with $\tilde\Omega$ and with $\left(I_{2n}\pm\frac{t}{2}\tilde\Omega\right)^{-1}$, respectively, and $\cay(-\frac{t}{2}\tilde\Omega)\cay(\frac{t}{2}\tilde\Omega)=I_{2n}$. Since $\gamma(t)^+=U^+\cay(-\frac{t}{2}\tilde\Omega)$, it follows that
        \begin{displaymath}
            \gamma(t)^+\dot\gamma(t)=\frac12 U^+\left(\left(I_{2n}+\frac{t}{2}\tilde\Omega\right)^{-1} + \left(I_{2n}-\frac{t}{2}\tilde\Omega\right)^{-1}\right)\tilde\Omega U.
        \end{displaymath}
    Furthermore $\gamma(t)^+\dot\gamma(t)=0$ is equivalent to $U\gamma(t)^+\dot\gamma(t)=0$.
   It holds that $\tilde \Omega \in \sp_{UU^+}(2n,\R)$,  since $\Delta \in \Hor^{\rho,h}_U\SpSt(2n,2k)$, which means $\tilde \Omega = \tilde \Omega UU^+ + UU^+ \tilde \Omega$. Then
    \begin{equation*}
        UU^+(I_{2n} \pm \frac{t}{2}\tilde\Omega)= (I_{2n}\mp \frac{t}{2}\tilde\Omega)UU^+
        \pm \frac{t}{2}\tilde\Omega,
    \end{equation*}
    which implies
    \begin{equation*}
        UU^+(I_{2n} \pm \frac{t}{2}\tilde\Omega)^{-1} = (I_{2n}\mp \frac{t}{2}\tilde\Omega)^{-1}UU^+ \mp \frac{t}{2} (I_{2n}\mp \frac{t}{2}\tilde\Omega)^{-1}\tilde\Omega(I_{2n}\pm\frac{t}{2}\tilde\Omega)^{-1}. 
    \end{equation*}
    Therefore 
    \begin{align*}
            U\gamma(t)^+\dot\gamma(t)
            &=\frac12 UU^+\left(\left(I_{2n}+\frac{t}{2}\tilde\Omega\right)^{-1} + \left(I_{2n}-\frac{t}{2}\tilde\Omega\right)^{-1}\right)\tilde\Omega U\\
            &= \frac12 \left(\left(I_{2n}-\frac{t}{2}\tilde\Omega\right)^{-1} + \left(I_{2n}+\frac{t}{2}\tilde\Omega\right)^{-1}\right)UU^+\tilde\Omega U
            =0,
    \end{align*}
    because  $UU^+\tilde\Omega U  = \tilde\Omega(I_{2n}-UU^+)U = 0$.
    \end{proof}
\end{proposition}

Projecting the retraction from Proposition \ref{prop:horizontalretraction} to the symplectic Grassmann manifold leads to a retraction on $\SpGr(2n,2k)$.
\begin{proposition}\label{prop:Grassmann_Retraction}
    Let $P \in \SpGr(2n,2k)$ and $\Gamma \in T_P\SpGr(2n,2k)$. Then
    \begin{equation*}
        \Ret^\SpGr_P(\Gamma):= \cay\left(\frac12[\Gamma,P]\right)P\cay\left(-\frac12[\Gamma,P]\right)
    \end{equation*}
    defines a retraction on $\SpGr(2n,2k)$. For $U \in \rho^{-1}(P) \subset \SpSt(2n,2k)$, it fulfills $\Ret^\SpGr_P(\Gamma) = \rho(\Ret^\SpSt_U(\Gamma^\hor_U))$. The curve $\gamma(t):=\Ret^\SpGr_P(t\Gamma)$ fulfills
    \begin{equation*}
        \dot \gamma(t) = \left[\frac12[\Gamma,P]\left(\left(I_{2n}-\frac{t}{2}[\Gamma,P]\right)^{-1} + \left(I_{2n}+\frac{t}{2}[\Gamma,P]\right)^{-1}\right), \gamma(t)\right].
    \end{equation*}
\end{proposition}
\begin{proof}
    The first retraction property $\Ret^\SpGr_P(0)=P$ is immediate. The formula for $\dot\gamma(t)$ follows from a direct calculation, whence $\D (\Ret^\SpGr_P)_0(\Gamma)=\frac{\D}{\D t}\Ret^\SpGr_P(t\Gamma)\vert_{t=0}=\dot\gamma(0)=[[\Gamma,P],P]=\Gamma$. This implies the second retraction property. 
\end{proof}

Similarly to Proposition \ref{prop:GrassmannLog}, we can invert the retraction on the symplectic Grassmann manifold in closed form. As in the symplectic Stiefel case, this defines local coordinates on $\SpGr(2n,2k)$.
In the following results, let $\sqrtm(\cdot)$ denote the principal matrix square root.

\begin{proposition}\label{prop:Inv_Grassmann_Retraction}
    Let $P,F \in \SpGr(2n,2k)$. If for
    \begin{equation}
    \label{eq:Inv_Grassmann_Retraction}
        \tilde\Omega := 2 \cay^{-1}\left(\sqrtm\left((I_{2n}-2F)(I_{2n}-2P)\right)\right),
    \end{equation}
    it holds that $\tilde\Omega \in \sp_P(2n)$, then $F = \Ret^\SpGr_P([\tilde\Omega,P])$.
\end{proposition}
\begin{proof}
    It holds that $\cay\left(\frac12\tilde\Omega\right)^2 = (I_{2n}-2F)(I_{2n}-2P)$.
    Since $(I_{2n}-2P)^2=I_{2n}$ and 
    \begin{equation*}
        (I_{2n}-2P)\cay\left(\frac12\tilde\Omega\right)(I_{2n}-2P)=\cay\left(-\frac12\tilde\Omega\right),
    \end{equation*}
    it follows that
    \begin{equation*}
    \begin{split}
        I_{2n}-2F &= \cay\left(\frac12\tilde\Omega\right)^2(I_{2n}-2P) = \cay\left(\frac12 \tilde\Omega\right)(I_{2n}-2P)\cay\left(-\frac12 \tilde\Omega\right)\\
        &=I_{2n}-2\cay\left(\frac12 \tilde\Omega\right)P\cay\left(-\frac12 \tilde\Omega\right),
    \end{split}
    \end{equation*}
    which implies the claimed result.
\end{proof}

We can also directly invert $\Ret^\SpGr_P(\Gamma)$ on symplectic Stiefel representatives. 

\begin{proposition}\label{prop:Inv_Grassmann_Retraction_lifted}
    Let $U, V \in \SpSt(2n,2k)$. If
    \begin{equation*}
        N:= (U^+V)^{-1}\sqrtm(U^+VV^+U) \in \Sp(2k,\R)
    \end{equation*}
    and
    \begin{equation*}
        H := 2(VN+U)(U^+VN + I_{2k})^{-1} - 2U \in \Hor^{\rho,h}_U\SpSt(2n,2k)
    \end{equation*}
    are well-defined, it holds that
    \begin{equation*}
        \Ret^\SpSt_U(H) = VN.
    \end{equation*}
\end{proposition}
\begin{proof}
    Since $NN^+ = (U^+V)^{-1}(U^+VV^+U)(V^+U)^{-1}=I_{2k}$, it holds that $N \in \Sp(2k,\R)$. Furthermore
    \begin{align*}
        N^+V^+U = (U^+VN)^+ = (\sqrtm(U^+VV^+U))^+ = \sqrtm(U^+VV^+U) = U^+VN.
    \end{align*}
    Therefore $\frac14 H^+H = 2(I_{2k}+U^+VN)^{-1} -I_{2k}$, which implies
    \begin{align*}
        \Ret^\SpSt_U(H) = -U + (H+2U)(\frac14 H^+H + I_{2k})^{-1} = -U + VN + U = VN.
    \end{align*}
\end{proof}

The difference between the connecting curves from Proposition~\ref{prop:Inv_Stiefel_Retraction} and Proposition~\ref{prop:Inv_Grassmann_Retraction_lifted} is visualized in Figure~\ref{fig:connectingCurves}.

\begin{figure}
 \centering
 \includegraphics{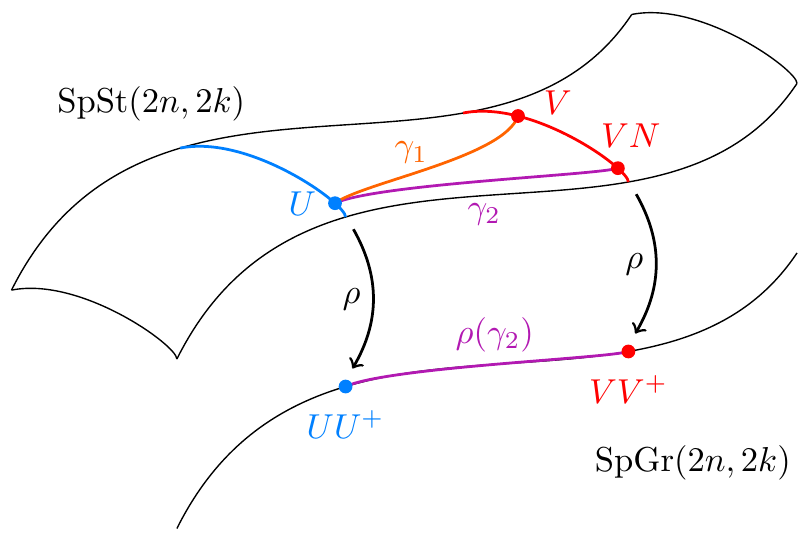}
 \caption{Connecting curves according to the inverse retractions Proposition~\ref{prop:Inv_Stiefel_Retraction} ($\gamma_1$) and Proposition~\ref{prop:Inv_Grassmann_Retraction_lifted} ($\gamma_2$).}
 \label{fig:connectingCurves}
\end{figure}

\section{Numerical Experiments}
\label{sec:NumericalTests}
In this section, we study the feasibility of different retractions on $\SpSt(2n,2k)$ and investigate optimization problems via gradient descent on $\SpSt(2n,2k)$ and $\SpGr(2n,2k)$, respectively.
All experiments are conducted with MATLAB version R2021a on a laptop with Ubuntu 18.04, an Intel® Core™ i7-8850H CPU and 16GB RAM.
We generate random Hamiltonian matrices via $\Omega = \begin{bsmallmatrix} A & B\\ C & -A^T \end{bsmallmatrix}$, where $A, B$ and $C$ are generated by \verb|randn(n,n)|, and then $B$ and $C$ are symmetrized.
For reproducability, all random matrices are constructed with the random stream \verb|s = RandStream('mt19937ar')|.

\subsection{Feasibility of different retractions on the symplectic Stiefel manifold}
\label{subsec:feasibility}

We compare the numerical feasibility of the Riemannian geodesic \eqref{eq:SpStGeodesicSmall} with respect to $g^\SpSt$, the Cayley-retraction \eqref{eq:RSpStSmall}, the pseudo-Riemannian geodesic \eqref{eq:StiefelReducedExp1} and the quasi-geodesic retraction \eqref{eq:QuasiGeodesicRetraction} from \cite{GaoSonAbsilStykel2020symplecticoptimization}. To this end, we generate a (pseudo) random point on $\SpSt(2n,2k)$ via $U = \cay(\Omega)E$, where $\Omega \in \sp(2n,\R)$ is scaled to $||\Omega||_F=1$. 
We furthermore generate a (pseudo) random tangent vector $\Delta \in T_U\SpSt(2n,2k)$, also scaled to $||\Delta||_F = 1$. For the chosen retractions $\Ret$, we calculate $U(t) = \Ret_U(t\Delta)$ with $t \in [0, 10^3]$ and plot the feasibility $||U(t)^+U(t) - I_{2n}||_F$. The average over $10$ runs is shown in Figure~\ref{fig:RetractionComparisonStiefel} for $n = 1000, k = 20$ (left) and $n = 1000, k = 200$ (right). It can be seen that the Riemannian geodesic, the pseudo-Riemannian geodesic and the quasi-geodesic retraction, which all rely on the matrix exponential, fail numerically to stay on $\SpSt(2n,2k)$ for tangent vectors of Frobenius-norm larger than $\mathcal{O}(10^2)$. The Cayley-Retraction, while less feasible at some points, fulfills the manifold condition up to an error of about $10^{-8}$ for tangent vectors of any tested size on $\SpSt(2000,40)$ and up to an error of about $10^{-4}$ on $\SpSt(2000,400)$.

\begin{figure}
\label{fig:RetractionComparisonStiefel}
\centering
 \includegraphics[width=.49\textwidth]{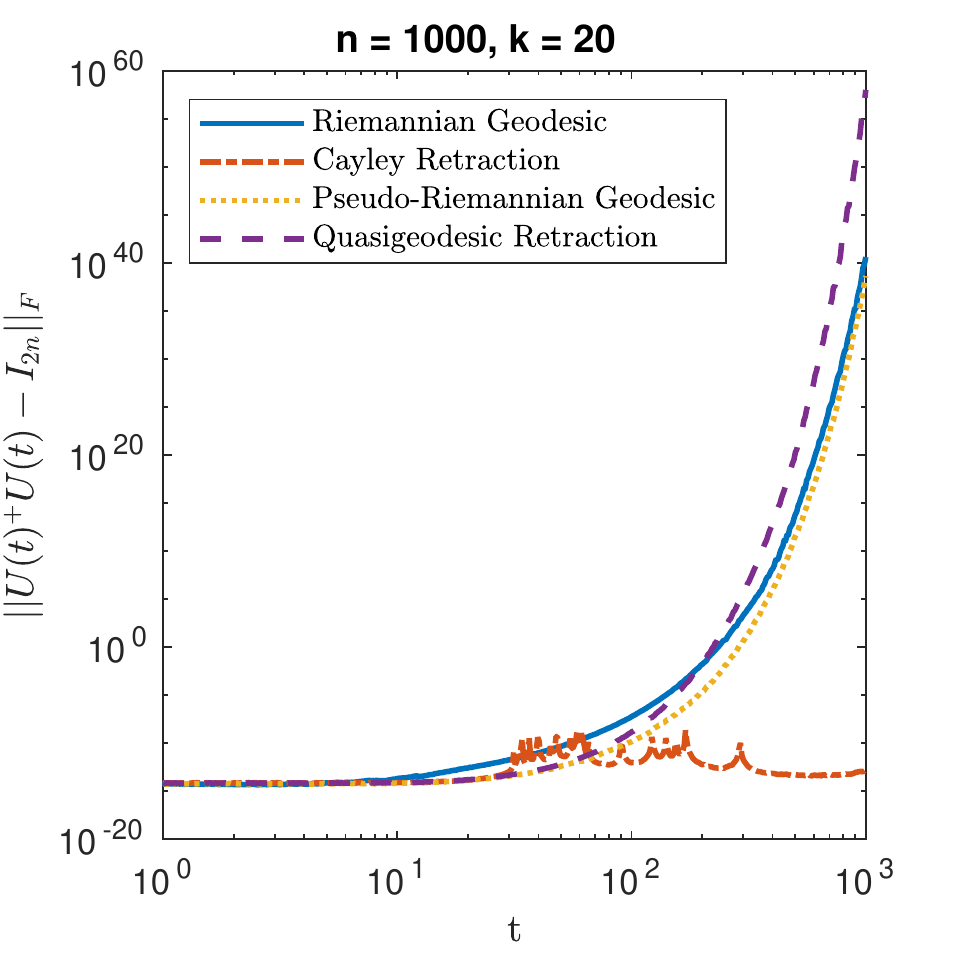}
 \includegraphics[width=.49\textwidth]{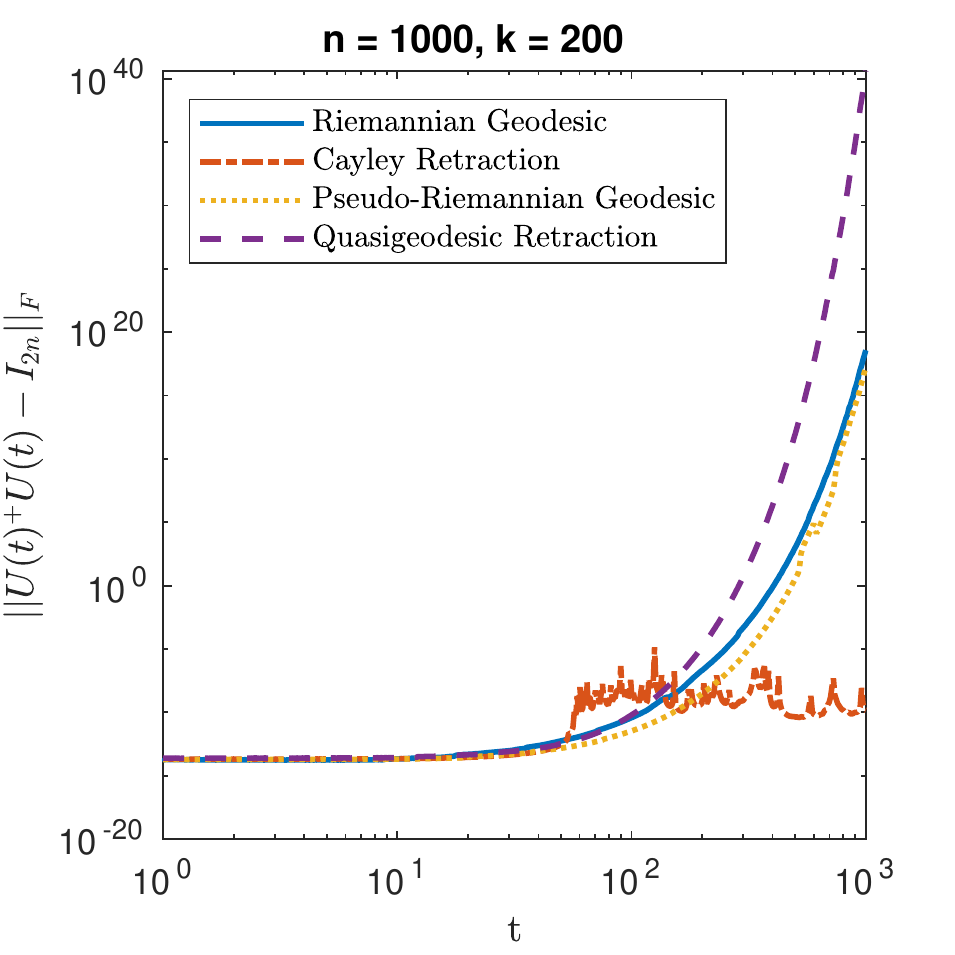}
 \caption{(cf. Subsection \ref{subsec:feasibility}) Comparison of the feasibility $\|U(t)^+U(t) - I_{2n}\|_F$ of different retractions $\Ret$ on $\SpSt(2000,40)$ and $\SpSt(2000,400)$, where $U(t) = \Ret_U(t\Delta)$ and $\|\Delta\|_F = 1$.  The data represent an average over $10$ runs, $U(t)$ is evaluated at $500$ logarithmically spaced steps.}
\end{figure}

\subsection{Gradient descent on the real symplectic Stiefel manifold}
\label{subsec:GradientDescentStiefel}
We tackle an academic instance of the `nearest symplectic matrix' problem
\begin{equation*}
    \min_{U \in \SpSt(2n,2k)}||U-A||_F^2
\end{equation*}
 via a Riemannian gradient descent. For this, we set \verb|A = randn(2*n,2*k)| and then normalize \verb|A = A/norm(A,2)|, as in \cite{GaoSonAbsilStykel2020symplecticoptimization}. The initial point for starting the optimization procedure is set to be \verb|U0 = cay(X/2)E|, where $X \in \sp(2n,\R)$ is a random $2n \times 2n$ Hamiltonian matrix, scaled by \verb|X = X/norm(X,'fro')|.
 As gradient descent algorithm we use \cite[Alg. 1]{GaoSonAbsilStykel2020symplecticoptimization} with monotone line search and stopping criterion \cite[Eqs. (6.1) and (6.2)]{GaoSonAbsilStykel2020symplecticoptimization}. For the reader's convenience, we restate the procedure here as Algorithm~\ref{alg:GradientDescent} in the precise form in which we use it.
\begin{algorithm}
    \caption{Gradient descent algorithm \cite[Alg. 1]{GaoSonAbsilStykel2020symplecticoptimization}}
    \label{alg:simple_quasigeodesic}
    \begin{algorithmic}[1]
        \Require $U_0 \in \SpSt(2n,2k)$, $f\colon \SpSt(2n,2k) \to \R$, retraction $\Ret$, $\beta,\delta \in (0,1)$, $0<\gamma_{\min}<\gamma_{\max}$, initial step size $\gamma_0^{ABB} = f(U_0)$, maximal iterations $N \in \N$, Riemannian metric $\SP{\cdot,\cdot}_U$ with gradient $\grad_f$, step parameters $h_{\min} < h_{\max} \in \Z$, tolerance parameters $\epsilon, \epsilon_x, \epsilon_f > 0$
        \For{$0\leq k \leq N$}
            \State $\Delta_k = -\grad_f(U_k)$
            \If{$k>0$}
            \State $S_k = U_k-U_{k-1}$ and $Y_k = \grad_f(U_k)-\grad_f(U_{k-1})$
            \If{$k$ is odd}
            \State $\gamma_k^{ABB} = \frac{\SP{S_k,S_k}}{|\SP{S_k,Y_k}|}$, where $\SP{\cdot,
            \cdot}$ denotes the Euclidean inner product.
            \Else
            \State $\gamma_k^{ABB} = \frac{|\SP{S_k,Y_k}|}{\SP{Y_k,Y_k}}$, where $\SP{\cdot,
            \cdot}$ denotes the Euclidean inner product.
            \EndIf
            \EndIf
            \State $\gamma_k = \max(\gamma_{\min},\min(\gamma_k^{ABB},\gamma_{\max}))$.
            \For{$h_{\min} \leq h \leq h_{\max}$}
                \State $t_k= \gamma_k \delta^h$
                \If{$f(\Ret_{U_k}(t_k\Delta_k))\leq f(U_k) - \beta t_k \SP{\Delta_k,\Delta_k}_{U_k}$}
                \State Break
                \EndIf
            \EndFor
            \State $U_{k+1} = \Ret_{U_k}(t_k\Delta_k)$
            \If{$||\grad_f(U_k)||_F < \epsilon$}
                \If{$\frac{|f(U_k) - f(U_{k+1})|}{|f(U_k)|+1} < \epsilon_f$ and $\frac{||U_k - U_{k+1}||_F}{\sqrt{2n}} < \epsilon_x$}
                \State Break as converged.
                \EndIf
            \EndIf
        \EndFor
        \Ensure Iterates $U_k$.
    \end{algorithmic}
    \label{alg:GradientDescent}
\end{algorithm}
 As the trial step size $\gamma_k$, we use the alternating BB method $\gamma_k^{ABB}$  \cite[Equation (6.4)]{GaoSonAbsilStykel2020symplecticoptimization}, with the respective gradient. The other method parameters are set to $\delta = 0.1, \beta = 10^{-4}, \gamma_{\text{min}}=10^{-15}$ and $\gamma_{\text{max}}=10^{15}$, as in \cite[Subsection 6.1]{GaoSonAbsilStykel2020symplecticoptimization}. The step parameters are set to $h_{\min} = 0$ and $h_{\max} = 5$, and the tolerance parameters to $\epsilon = 10^{-6}$, $\epsilon_x = 10^{-6}$ and $\epsilon_f = 10^{-12}$, respectively.

We compare gradient descent for the Riemannian metric $g^\SpSt$ with geodesic stepping and Cayley stepping, respectively, with gradient descent from \cite{GaoSonAbsilStykel2020symplecticoptimization} with Cayley stepping. For the gradient descent according to \cite{GaoSonAbsilStykel2020symplecticoptimization}, we choose the optimal settings stated in this reference, i.e., the canonical-like metric $g_\rho$ with $\rho = \frac12$ and gradient (I), according to \cite[Subsection 6.2.2]{GaoSonAbsilStykel2020symplecticoptimization}.
In the actual implementation of all methods included in this comparison, care has been taken that the action of large matrices like $J_{2n}$ and $I_{2n}$ is applied directly,
so that these matrices are never formed explicitly.

Figure \ref{fig:GradientDescentStiefel} displays the objective function value versus the iteration count (left) and the convergence history according to the gradient norm (right), respectively. For comparison purposes, all methods are run for a fixed number of $60$ iterations. It can be seen that the algorithms deliver similar results in regard of the convergence by iterations, depending on the chosen tolerance. The run time however differs: In Table~\ref{tab:numEx1}, we compare the three methods and state the average iterations and run time until numerical convergence over $10$ runs. We furthermore denote the relative deviation from the respective minimum over all three methods after convergence. It can be seen that for $\SpSt(2000,40)$ and $\SpSt(2000,400)$, gradient descent with Cayley stepping is the fastest method regarding run time, while
Geodesic descent is the slowest. For $\SpSt(2000,400)$, the run time for geodesic stepping increases drastically, since \eqref{eq:SpStGeodesicSmall} requires the matrix exponential of both a $8k \times 8k$ and a $4k \times 4k$ matrix.
\begin{table}[!ht]
\renewcommand{\arraystretch}{1.3}
\begin{small}
\begin{center}
\caption{Numerical performance for the cases considered in Section \ref{subsec:GradientDescentStiefel}, taking averages over $10$ runs. The (pseudo-)random data are generated for $n=1000$, and either $k=20$ or $k=200$, respectively. The minimum is the respective minimum over all three methods after convergence.
}
\begin{tabular}{l l >{\columncolor{gray!10}}l l >{\columncolor{gray!10}}l l >{\columncolor{gray!10}}l}
 \toprule
 & $k=20$ & $200$ & $20$ & $200$ & $20$ & $200$\\
  \cmidrule{2-3} \cmidrule{4-5} \cmidrule{6-7} Method & \multicolumn{2}{l}{rel. deviation from minimum} & \multicolumn{2}{l}{iterations} & \multicolumn{2}{l}{run time (s)}  \\
 \midrule
 $g^{SpSt}$, Geodesic &     $1.6614\cdot 10^{-15}$ & $5.4229 \cdot 10^{-16}$ & $25.5$ & $46.5$ & $0.4091$ s & $46.4378$ s \\
 $g^{SpSt}$, Cayley   &     $3.8021\cdot 10^{-15}$ & $1.1783 \cdot 10^{-15}$ & $25.4$ & $48.6$ & $0.2008$ s & $9.1229$ s\\
 $g_\rho$ from \cite{GaoSonAbsilStykel2020symplecticoptimization}           &     $6.8733\cdot 10^{-14}$ & $1.1445 \cdot 10^{-14}$ & $32.1$ & $42.5$ & $0.3098$ s & $11.3064$ s \\
 \bottomrule
%
\end{tabular}
\label{tab:numEx1}
\end{center}
\end{small}
\end{table}

In Figure~\ref{fig:GradientDescentStiefeltime}, we compare the convergence over time for one optimizer run on $\SpSt(2000,40)$. For each step, the run time is measured over one full iteration of the outer for-loop in lines 1 to 24 in Algorithm~\ref{alg:GradientDescent}. It can be seen that gradient descent with respect to the Riemannian metric $g^\SpSt$ with Cayley stepping converges the fastest in terms of the  run time. The iteration count for Cayley and geodesic stepping with respect to $g^\SpSt$ are comparable.

For Figure~\ref{fig:GradientDescentStiefeltime2A}, we repeat the experiment from Figure~\ref{fig:GradientDescentStiefeltime} with the setting
featured in \cite[Figure 6]{GaoSonAbsilStykel2020symplecticoptimization}, i.e., we
scale $A$ to \verb|A = 2*A/norm(A,2)|. In this case, the iteration count until convergence stays approximately the same for gradient descent with respect to the quotient metric $g^\SpSt$, while it increases considerably for the canonical-like metric $g_\rho$.

\begin{figure}
\label{fig:GradientDescentStiefel}
 \includegraphics[width=\textwidth]{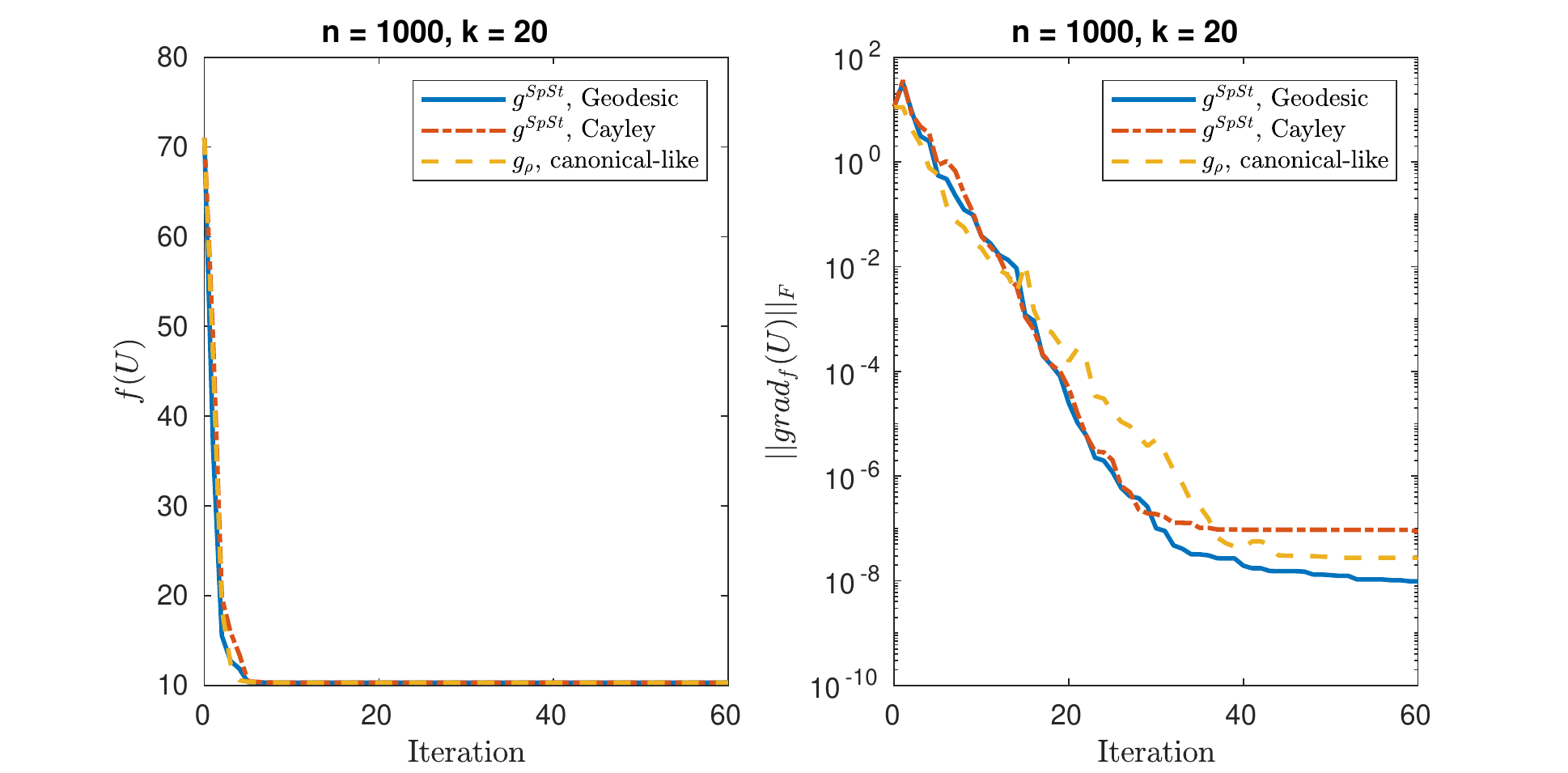}
 \caption{(cf. Subsection \ref{subsec:GradientDescentStiefel}) Comparison of Riemannian gradient descent on $\SpSt(2000,40)$ to find the symplectic Stiefel matrix closest to a random matrix $A$. The data represent an average over $10$ runs. Here, $g_\rho$ denotes the canonical-like metric from \cite{GaoSonAbsilStykel2020symplecticoptimization} with Cayley stepping.}
\end{figure}

\begin{figure}
\label{fig:GradientDescentStiefeltime}
 \includegraphics[width=\textwidth]{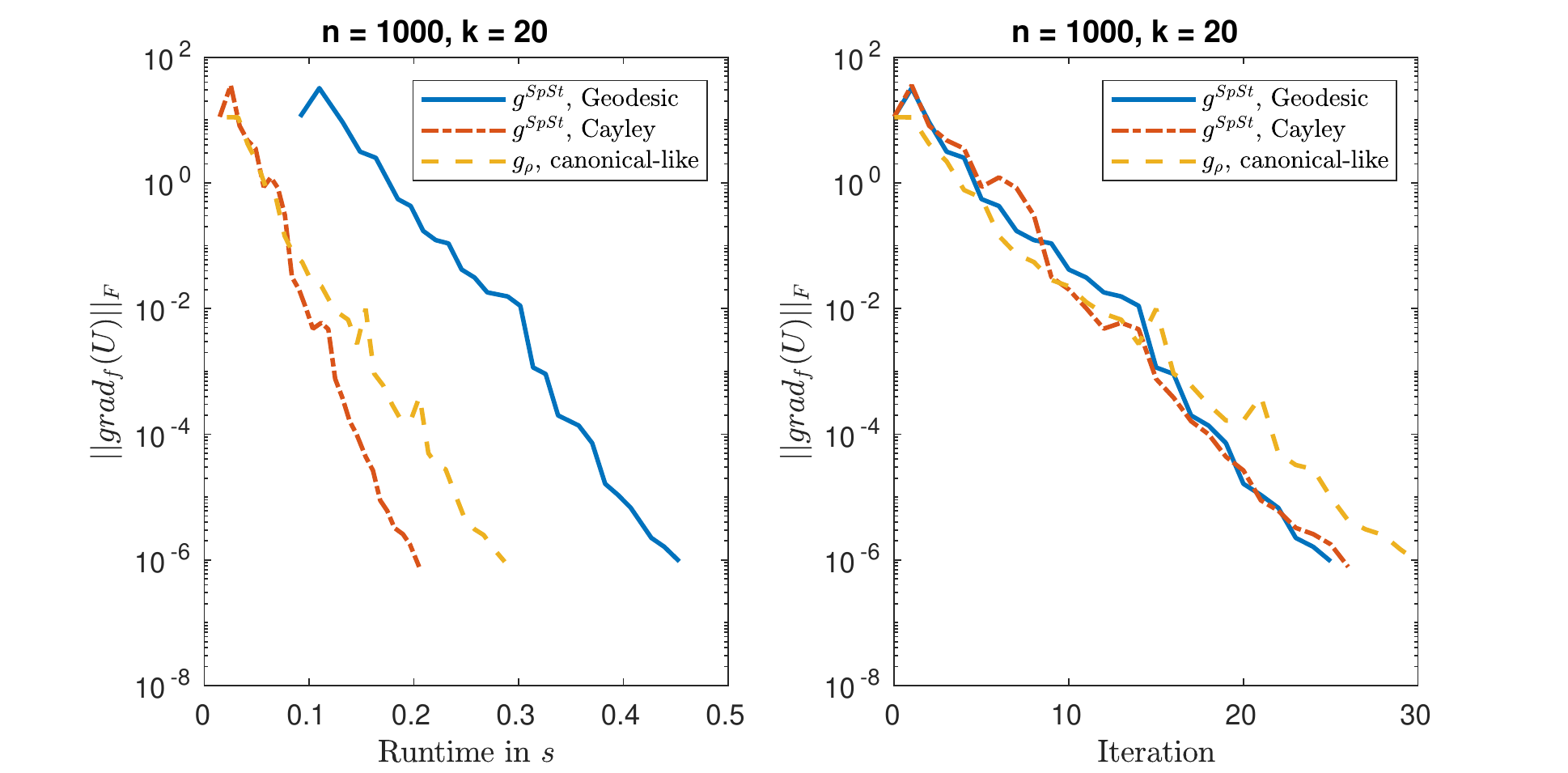}
 \caption{(cf. Subsection \ref{subsec:GradientDescentStiefel}) Comparison of Riemannian gradient descent on $\SpSt(2000,40)$ to find the symplectic Stiefel matrix closest to a random matrix $A$ versus time. Here, $g_\rho$ denotes the canonical-like metric from \cite{GaoSonAbsilStykel2020symplecticoptimization} with Cayley stepping.}
 \end{figure}
 
 \begin{figure}
\label{fig:GradientDescentStiefeltime2A}
 \includegraphics[width=\textwidth]{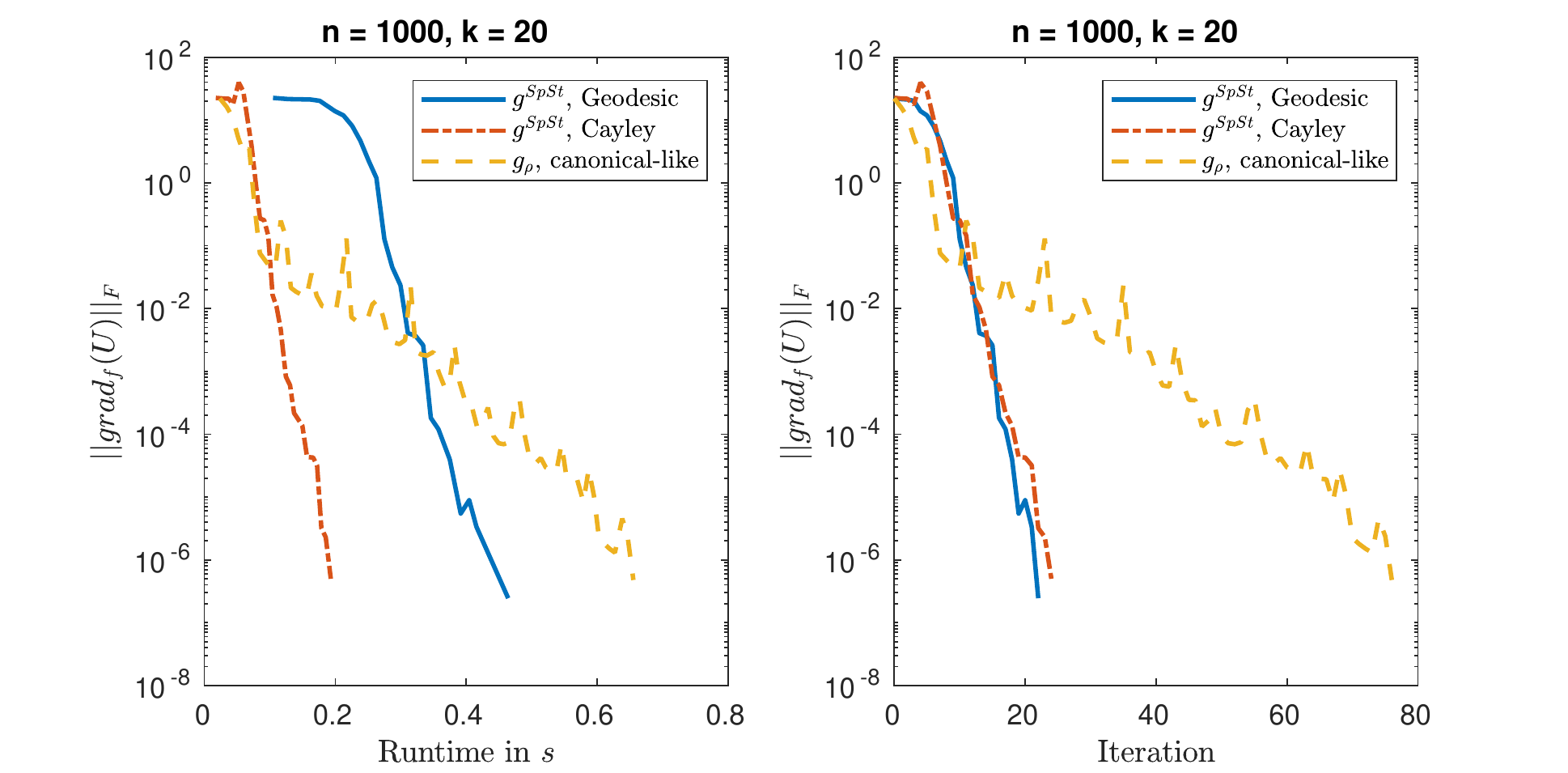}
 \caption{(cf. Subsection \ref{subsec:GradientDescentStiefel}) 
 Analogous to  Figure \ref{fig:GradientDescentStiefeltime} but with the target matrix $A$ scaled to $||A||_2 = 2$. Here, $g_\rho$ denotes the canonical-like metric from \cite{GaoSonAbsilStykel2020symplecticoptimization} with Cayley stepping.}
 \end{figure}

\subsection{Gradient descent on the real symplectic Grassmann manifold}
\label{subsec:GradientDescentGrassmann}
In this subsection, we consider optimization via gradient descent 
on the real symplectic Grassmann manifold.
More precisely, we search for the optimal symplectic subspace for representing a 
given data matrix $S \in \R^{2n \times 2n}$, i.e.,
\begin{equation}
\label{eq:min_SpGr}
 \min_{U \in \SpSt(2n,2k)} ||S - UU^+S ||_F^2.
\end{equation}
This problem is associated with computing a \emph{proper symplectic decomposition},
a task that is central in Hamiltonian model order reduction \cite{peng2016symplectic}. 
Here, we work in an academic setting, where the target matrix $S$ is generated as a random symplectic subspace representative plus an error term, i.e. 
\begin{equation*}
    S = AA^+ + \mathcal{E},
\end{equation*} 
where $A \in \SpSt(2n,2k)$ is a random symplectic Stiefel matrix found in the same manner as the initial point $U_0$, and $\mathcal{E}$ is a random $n\times n$-matrix, divided by its 2-norm. The parameters for the gradient descent algorithm are the same as in Subsection~\ref{subsec:GradientDescentStiefel}. The resulting average of the function value and the convergence history over $10$ runs with a fixed number of $40$ iterations is shown in Figure~\ref{fig:GradientDescentGrassmann}. It can be seen that gradient descent for all methods produces similar results in regards of the iteration count. For the gradient descent from \cite{GaoSonAbsilStykel2020symplecticoptimization} and for the gradient descent according to $g^\SpSt$ with Cayley stepping, we ignore the quotient structure and treat \eqref{eq:min_SpGr} as a minimization problem on $\SpSt(2n,2k)$. The run time and iteration count of the methods is compared in Table~\ref{tab:numEx2}, similarly to Subsection~\ref{subsec:GradientDescentStiefel}. We also compare the convergence over run time for a single optimizer run in Figure~\ref{fig:GradientDescentGrassmanntime}. It can be seen that gradient descent with Cayley stepping according to $g^\SpSt$ or $g^\SpGr$ converges fastest and both methods perform comparable to the method of \cite{GaoSonAbsilStykel2020symplecticoptimization}. As is to be expected, for $k = 200$,  geodesic stepping is again considerably slower. Note however that for all methods, processing the $2n \times 2n$ input matrix $S$ requires a high base level run time.

\begin{remark}
    It is also possible to tackle optimization problems with pseudo-Riemannian methods \cite{fiori2011,GaoLekHengYe2018semiriemannian}.
    For our experiments however, we achieved better result with the Riemannian methods.
    Nevertheless, pseudo-Riemannian optimization might prove beneficial in some settings.
\end{remark}

\begin{figure}
\label{fig:GradientDescentGrassmann}
 \includegraphics[width=\textwidth]{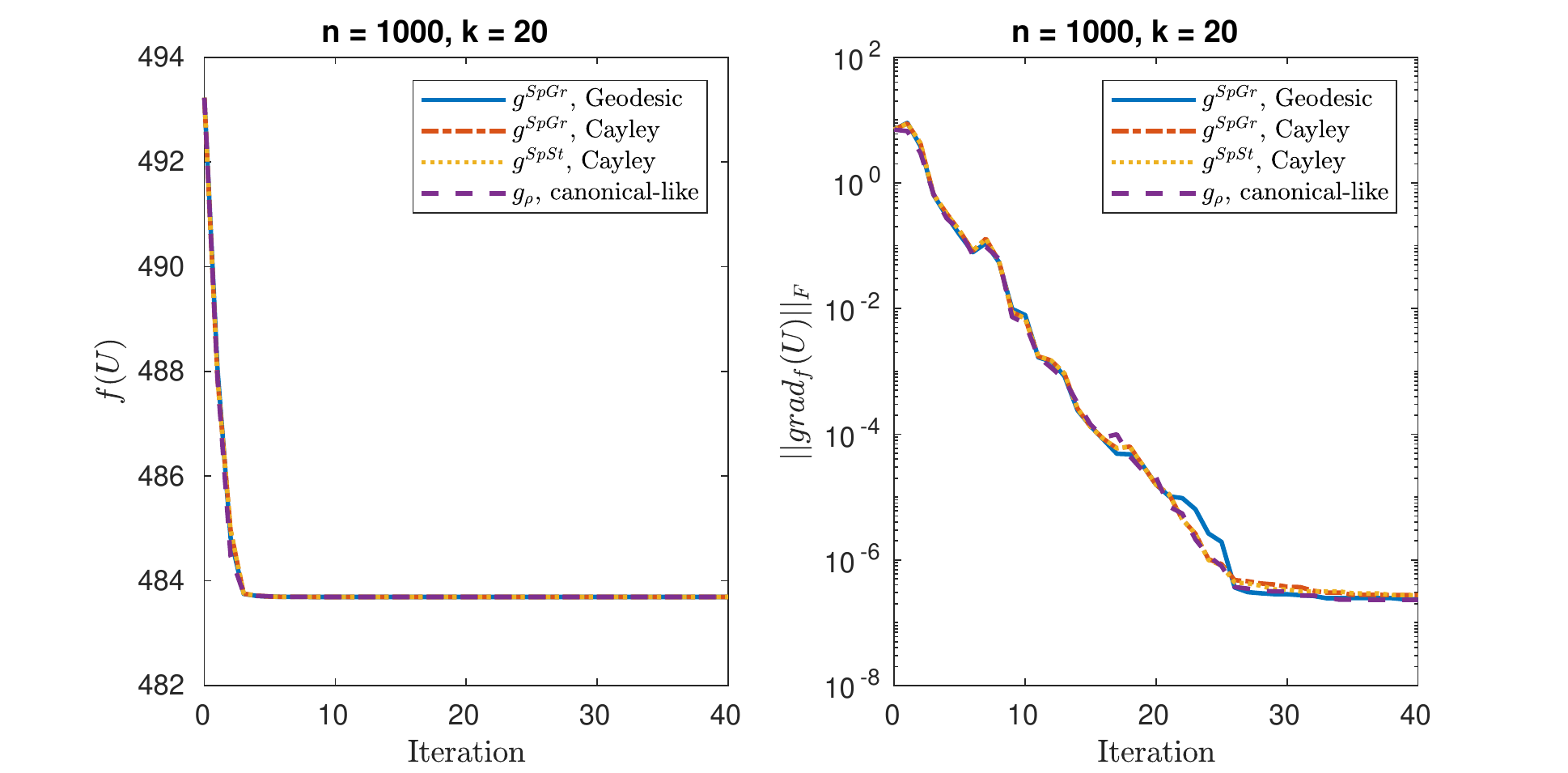}
 \caption{(cf. Subsection \ref{subsec:GradientDescentGrassmann}) Comparison of Riemannian gradient descent on $\SpGr(2000,40)$ to find the optimal subspace representing a matrix $S$. The data represent an average over $10$ runs. Here, $g_\rho$ denotes the canonical-like metric from \cite{GaoSonAbsilStykel2020symplecticoptimization} with Cayley stepping.}
\end{figure}

\begin{figure}
\label{fig:GradientDescentGrassmanntime}
 \includegraphics[width=\textwidth]{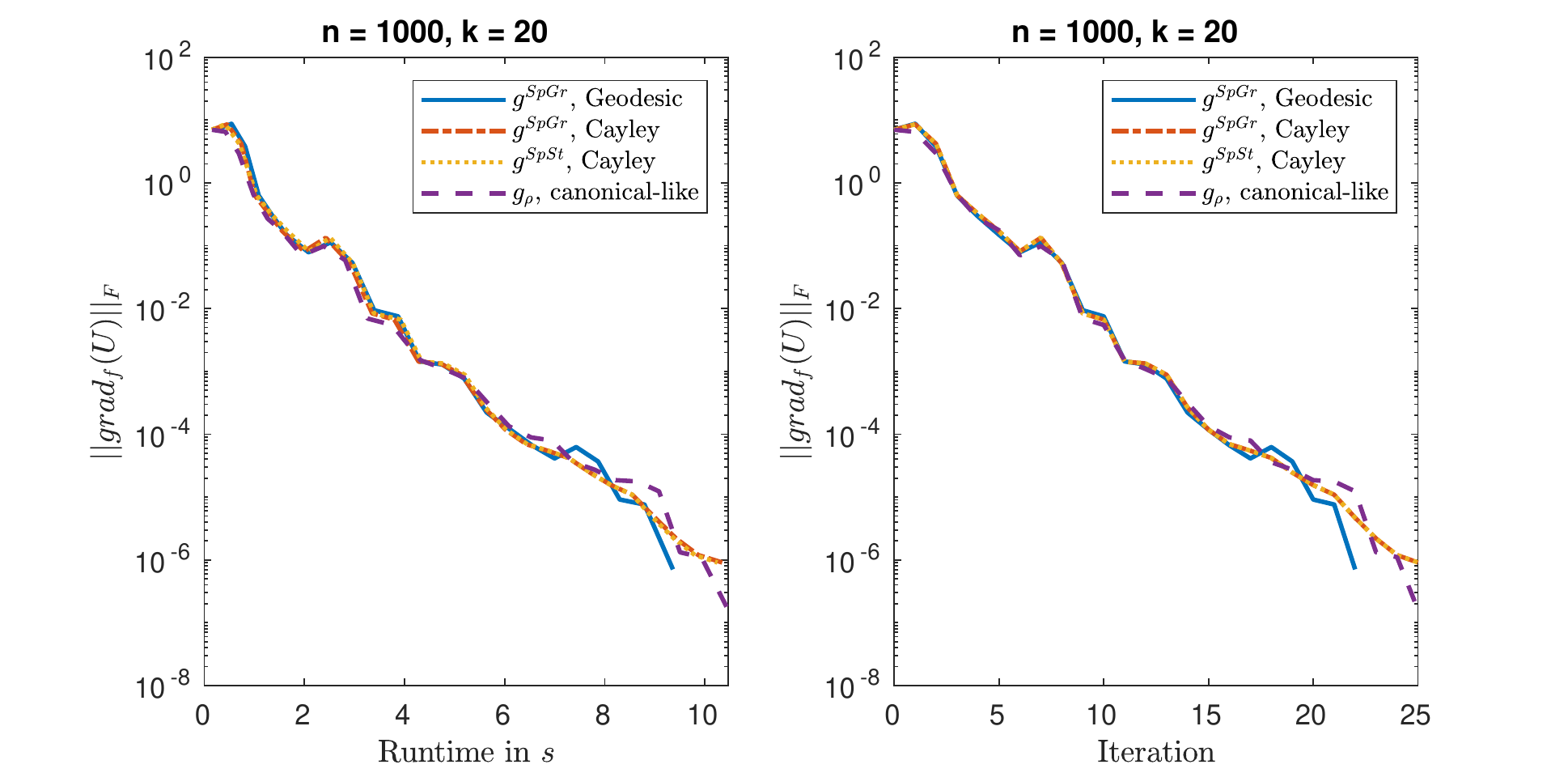}
 \caption{(cf. Subsection \ref{subsec:GradientDescentGrassmann}) Comparison of Riemannian gradient descent on $\SpGr(2000,40)$ over time to find the optimal subspace representing a matrix $S$. Here, $g_\rho$ denotes the canonical-like metric from \cite{GaoSonAbsilStykel2020symplecticoptimization} with Cayley stepping.}
\end{figure}

\begin{table}[!ht]
\renewcommand{\arraystretch}{1.3}
\begin{small}
\begin{center}
\caption{Numerical performance for the cases considered in Section \ref{subsec:GradientDescentGrassmann}, taking averages over $10$ runs. The (pseudo-)random data are generated for $n=1000$, and either $k=20$ or $k=200$, respectively. The minimum is the respective minimum over all five methods after convergence.
}
\begin{tabular}{l l >{\columncolor{gray!10}}l l >{\columncolor{gray!10}}l l >{\columncolor{gray!10}}l}
 \toprule
 & $k=20$ & $200$ & $20$ & $200$ & $20$ & $200$\\
  \cmidrule{2-3} \cmidrule{4-5} \cmidrule{6-7} Method & \multicolumn{2}{l}{rel. deviation from minimum} & \multicolumn{2}{l}{iterations} & \multicolumn{2}{l}{run time (s)}  \\
 \midrule
 $g^{SpGr}$, Geod. &     $7.0799\cdot 10^{-17}$ & $6.9878\cdot 10^{-17}$ & $25.2$ & $33.7$ & $11.4656$ & $61.9820$ \\
 $g^{SpGr}$, Cayley   &     $2.1071\cdot 10^{-16}$ & $1.0422\cdot 10^{-16}$ & $23.9$ & $35.4$ & $10.5513$ & $36.5796$\\
 $g^{SpSt}$, Geod. &     $1.0617\cdot 10^{-16}$ & $1.9201\cdot 10^{-16}$ & $25.2$ & $34.1$ &  $11.4330$ & $64.3123$ \\
 
  $g^{SpSt}$, Cayley &     $2.1071\cdot 10^{-16}$ & $1.0422\cdot 10^{-16}$ & $23.9$  & $35.4$ & $10.5545$ & $35.3957$ \\

 $g_\rho$ from \cite{GaoSonAbsilStykel2020symplecticoptimization}           &     $3.5375\cdot 10^{-17}$ & $1.9164\cdot 10^{-16}$ & $24.8$ & $34.7$ & $10.9243$ & $37.2282$ \\
 \bottomrule
%
\end{tabular}
\label{tab:numEx2}
\end{center}
\end{small}
\end{table}

\section{Conclusion}
\label{sec:Conc}
We introduced a novel pseudo-Riemannian framework for the real symplectic Stiefel manifold $\SpSt(2n,2k)$. In analogy to the classical Stiefel and Grassmann manifolds, we introduced the real symplectic Grassmann manifold $\SpGr(2n,2k)$. 
For a natural pseudo-Riemannian metric, we derived the corresponding geodesics.
With the formulas at hand, we explained the Cayley retraction as an approximation of the pseudo-Riemannian geodesics and found an efficiently computable expression for the retraction, which turned out to be invertible in closed form. 


Secondly, we introduced a new Riemannian framework for both $\SpSt(2n,2k)$ and $\SpGr(2n,2k)$, coming from a right-invariant Riemannian metric on $\Sp(2n,\R)$, and derived the corresponding Riemannian geodesics.
Since to the best of the authors' knowledge, the Riemannian geodesics for no other Riemannian metric on $\SpSt(2n,2k)$ are known, this opens up new possibilities for theoretical studies and applications.

In the experiments, we showed that gradient descent with the Riemannian geo\-de\-sics or optimized Cayley retraction outperforms the state-of-the-art method from \cite{GaoSonAbsilStykel2020symplecticoptimization} in some cases and delivers comparable results in others. Cayley stepping with respect to the Riemannian metric $g^\SpSt$ on $\SpSt(2n,2k)$ converges in general the fastest among all methods, regarding the run time.

The invertible retractions provide local coordinates on the manifolds  $\SpSt(2n,2k)$ and $\SpGr(2n,2k)$, respectively.
This renders it possible to apply tangent space methods, e.g. for interpolation purposes. 
A potential area of application of such tangent space interpolation
is parametric model order reduction of Hamiltonian systems.
The proposed coordinate transformations allow to approach this problem analogously to parametric model order reduction of general dynamical systems \cite{BennerGugercinWillcox2015,Zimmermann_MORHB2021}.


\bibliographystyle{siamplain}
\bibliography{symplbib}

\end{document}